\documentclass[12pt, a4paper]{amsart}
\usepackage{refcheck}
\usepackage{verbatim}
\usepackage[dvips]{color}
\usepackage{amssymb}
\usepackage[active]{srcltx}
\usepackage{latexsym}
\usepackage{graphicx}
\usepackage{color}
\usepackage{amsmath}
 \usepackage{float}
\usepackage{mathrsfs} %fonts veramente calligrafici con \mathscr{B}
\usepackage[draft]{fixme}

\usepackage{caption}

\allowdisplaybreaks

\newtheorem{theorem}{Theorem}[section]
\newtheorem{proposition}{Proposition}[section]
\newtheorem{lemma}{Lemma}[section]
\newtheorem{corollary}{Corollary}[section]
\theoremstyle{definition}
\newtheorem{definition}{Definition}
\newtheorem{remark}{Remark}[section]

\floatstyle{plain} \restylefloat{figure} \restylefloat{table}

\numberwithin{equation}{section}

\newcommand{\beq}{\begin{equation}}
\newcommand{\bea}[1]{\begin{array}{#1} }
\newcommand{\eeq}{ \end{equation}}
\newcommand{\ea}{ \end{array}}
\newcommand{\ep}{\epsilon}

\def \R  {{\mathbb {R}}}
\def \N  {{\mathbb {N}}}
\def \F  {{\mathscr {F}}}

\def \t {{\tau}}

\def\mean#1{\mathchoice%
          {\mathop{\kern 0.2em\vrule width 0.6em height 0.69678ex depth -0.58065ex
                  \kern -0.8em \intop}\nolimits_{\kern -0.4em#1}}%
          {\mathop{\kern 0.1em\vrule width 0.5em height 0.69678ex depth -0.60387ex
                  \kern -0.6em \intop}\nolimits_{#1}}%
          {\mathop{\kern 0.1em\vrule width 0.5em height 0.69678ex
              depth -0.60387ex
                  \kern -0.6em \intop}\nolimits_{#1}}%
          {\mathop{\kern 0.1em\vrule width 0.5em height 0.69678ex depth -0.60387ex
                  \kern -0.6em \intop}\nolimits_{#1}}}

\def\vintslides_#1{\mathchoice%
          {\mathop{\kern 0.1em\vrule width 0.5em height 0.697ex depth -0.581ex
                  \kern -0.6em \intop}\nolimits_{\kern -0.4em#1}}%
          {\mathop{\kern 0.1em\vrule width 0.3em height 0.697ex depth -0.604ex
                  \kern -0.4em \intop}\nolimits_{#1}}%
          {\mathop{\kern 0.1em\vrule width 0.3em height 0.697ex depth -0.604ex
                  \kern -0.4em \intop}\nolimits_{#1}}%
          {\mathop{\kern 0.1em\vrule width 0.3em height 0.697ex depth -0.604ex
                  \kern -0.4em \intop}\nolimits_{#1}}}

\newcommand{\aveint}[2]{\mathchoice%
          {\mathop{\kern 0.2em\vrule width 0.6em height 0.69678ex depth -0.58065ex
                  \kern -0.8em \intop}\nolimits_{\kern -0.45em#1}^{#2}}%
          {\mathop{\kern 0.1em\vrule width 0.5em height 0.69678ex depth -0.60387ex
                  \kern -0.6em \intop}\nolimits_{#1}^{#2}}%
          {\mathop{\kern 0.1em\vrule width 0.5em height 0.69678ex depth -0.60387ex
                  \kern -0.6em \intop}\nolimits_{#1}^{#2}}%
          {\mathop{\kern 0.1em\vrule width 0.5em height 0.69678ex depth -0.60387ex
                  \kern -0.6em \intop}\nolimits_{#1}^{#2}}}

  \newcommand{\e}{\epsilon}

\def\eqn#1$$#2$${\begin{equation}\label#1#2\end{equation}}
\def\charfn_#1{{\raise1.2pt\hbox{$\chi
_{\kern-1pt\lower3pt\hbox{{$\scriptstyle#1$}}}$}}}

\def\qq1{q_*}
\def\q2{q_{**}}

\def\ep{\varepsilon}

\newdimen\vintbar
\def\vint{-\kern-\vintbar\int}

\def\A{\mathcal A}

\def\E{\mathbb E}
\def\F{\mathcal F}

\def\H{\mathcal H}
\def\I{\mathcal I}

\def\N{\mathcal N}

\def\P{\mathbb P}
\def\Q{\mathcal Q}

\def\0{\boldsymbol 0}

\renewcommand{\l}{\left}
\renewcommand{\r}{\right}

\DeclareMathOperator*{\argmin}{arg\,min}

\newtoks\by
\newtoks\paper
\newtoks\book
\newtoks\jour
\newtoks\yr
\newtoks\pages
\newtoks\vol
\newtoks\publ

\def\name[#1, #2]{#1 #2}
\def\ota{{\hbox{\bf ???}}}
\def\cLear{\by=\ota\paper=\ota\book=\ota\jour=\ota\yr=\ota
\pages=\ota\vol=\ota\publ=\ota}
\def\endpaper{\the\by, \textit{\the\paper},
{\the\jour} \textbf{\the\vol} (\the\yr), \the\pages.\cLear}
\def\endbook{\the\by, \textit{\the\book},
\the\publ, \the\yr.\cLear}
\def\endpap{\the\by, \textit{\the\paper}, \the\jour.\cLear}
\def\endproc{\the\by, \textit{\the\paper}, \the\book, \the\publ,
\the\yr, \the\pages.\cLear}

\begin{document}

\title[OPTIMAL SWITCHING PROBLEMS UNDER PARTIAL INFORMATION]{OPTIMAL SWITCHING PROBLEMS UNDER\\ PARTIAL INFORMATION}

\author{K. Li, K. Nystr{\"o}m, M. Olofsson}\thanks{\noindent K. Li and M. Olofsson are financed
by Jan Wallanders och Tom Hedelius Stiftelse samt Tore Browaldhs Stiftelse
through the project {\it Optimal switching problems and their applications in economics and finance}, P2010-0033:1.}

%\author[Lundstr\"{o}m]{Kai Li}
\address{Kai Li \\Department of Mathematics, Uppsala University\\
S-751 06 Uppsala, Sweden}
\email{kai.li@math.uu.se}
%\author[Nystr\"{o}m]{Kaj Nystr\"{o}m}
\address{Kaj Nystr\"{o}m\\Department of Mathematics, Uppsala University\\
S-751 06 Uppsala, Sweden}
\email{kaj.nystrom@math.uu.se}
%\author[Olofsson]{Marcus Olofsson}
\address{Marcus Olofsson\\Department of Mathematics, Uppsala University\\
S-751 06 Uppsala, Sweden}
\email{marcus.olofsson@math.uu.se}
%}

\norefnames %Belongs to refcheck!
\nocitenames

\begin{abstract}
\noindent
In this paper we formulate and study an optimal switching problem under partial information. In our model the agent/manager/investor attempts to maximize the expected
reward by switching between different states/investments. However,  he is not fully aware of his environment and only an observation process, which contains partial information about the environment/underlying, is accessible. It is based on the partial information carried by this observation process that all decisions must be made. We propose
a probabilistic numerical algorithm based on dynamic programming, regression Monte Carlo methods, and stochastic filtering theory to compute the value function.
In this paper, the approximation of the value function and the corresponding convergence result are obtained when the underlying and observation processes satisfy the linear Kalman-Bucy setting. A numerical example is included to show some specific features of partial information.
\newline

\noindent

\medskip

\end{abstract}
%\newpage
\maketitle

%\tableofcontents
%\newpage

\setcounter{equation}{0} \setcounter{theorem}{0}
\section{Introduction}
\noindent
In recent years there has been an increasing activity in the study of optimal switching problems, associated reflected backward stochastic differential
equations and systems of variational inequalities, due to the potential use of these types of models/problems to address the problem of valuing investment opportunities, in an uncertain world, when the investor/producer is allowed to switch between different investments/portfolios or  production modes. To briefly
outline the traditional setting of multi-modes optimal switching problems, we consider a production facility which can be run in $d$ ($d\geq 2$) different production modes and assume that the running pay-offs in the different modes, as well as the cost of switching between modes, depend on an $N_1$-dimensional diffusion process
$X=\left\{X_s^{x,t}\right\}$ which is a solution to the system of stochastic differential equations
\begin{eqnarray}\label{e-SDE}
dX_s^{x,t}&=& b(X_s^{x,t},s) ds+\sigma(X_s^{x,t},s)dW_s,\  t \leq s \leq T,  \notag \\
X_{t}^{x,t}&=&x,
\end{eqnarray}
where $(x,t)\in\mathbb R^{N_1}\times[0,T]$ and $W=\{W_s\}$ is an $m_1$-dimensional Brownian motion, $m_1\leq
N_1$,  defined on a filtered probability space $(\Omega,\mathcal F, \{\mathcal F_t\}_{t \geq 0},\mathbb P)$. In the case of electricity and energy production the process
$X=\left\{X_s^{x,t}\right\}$ can, for instance, be the electricity price, functionals of the electricity price,  or other factors, like the national product or other indices measuring the state of the local and global business cycle, which in turn influence the price. Given $X=\left\{X_s^{x,t}\right\}$ as in \eqref{e-SDE}, let the payoff rate in production mode $i$, at time $s$, be $f_i(X^{x,t}_s,s)$ and let $c_{i,j}(X^{x,t}_s,s)$ be the continuous switching cost for switching from mode $i$ to mode $j$ at time $s$.  A management strategy is a combination of a non-decreasing sequence of $\mathcal F_s$-adapted stopping times $\{\tau_k\}_{k\geq 0}$, where, at time $\tau_k$, the manager decides to switch the production
from its current mode to another one, and a sequence of  $\mathcal F_s$-adapted indicators $\{\xi_k\}_{k\geq 0}$, taking values in $\{1,\dots,d\}$, indicating the mode to which the production is switched. At $\tau_k$  the production is switched from mode $\xi_{k-1}$ (current mode) to $\xi_k$.  When the production is run under a strategy $\mu=(\{\tau_k\}_{k\geq 0},\{\xi_k\}_{k\geq 0})$, over a finite horizon $[0,T]$, the total expected
profit is
defined as%
\begin{eqnarray*}%\label{eq1}
J(\mu):=\mathbb E\biggl [\biggl(\int\limits_0^Tf_{\mu_s}(X^{x,t}_s,s)ds-\sum_{0 \leq \tau_k\leq T}c_{\xi_{{k-1}},\xi_{k}}(X^{x,t}_{\tau_{k}},\tau_{k})\biggr )\biggr ],
\end{eqnarray*}
where $\mu=(\mu_s)$ is the to $\mu$ associated index process. The traditional multi-modes optimal switching problem now consists of finding an optimal management strategy
$\mu^\ast=(\{\tau_k^\ast\}_{k\geq 0},\{\xi_k^\ast\}_{k\geq 0})$ such that
\begin{eqnarray*}%\label{eq2}
J(\mu^\ast)=\sup_{\mu}J(\mu).
\end{eqnarray*}
Let from now on  $\mathcal F_s^X$ denote the filtration generated by the process $X$ up to time $s$, i.e., $\mathcal F_s^X=\sigma(X^{x,t}_u:0\leq u\leq s)$. We let $\mathcal A^X=\mathcal A^X[0,T]$ denote the set of all (admissible) strategies $\mu=(\{\tau_k\}_{k\geq 0},\{\xi_k\}_{k\geq 0})$ such that
$0\leq \tau_k \leq T$ for $k\geq 0$, and such that the stopping times $\{\tau_k\}_{k\geq 0}$ and the indicators $\{\xi_k\}_{k\geq 0}$ are adapted to the filtration
$\{\mathcal F_s^X\}_{\{0\leq s\leq T\}}$. Furthermore, given $t\in[0,T]$, $i\in\{1,\dots,d\}$, we let $\mathcal A_{t,i}^X\subset \mathcal A^X$, be the subset of strategies such that $\tau_1 \geq t$ and $\xi_0=i$ a.s. We let
\begin{eqnarray}\label{eq1+}
u_i(x,t)=\sup _{{\mu} \in \mathcal A_{t,i}^X} \mathbb E\biggl [\biggl(\int\limits_t^Tf_{\mu_s}(X_s^{x,t},s)ds- \sum _{ t \leq \tau_k \leq T}c_{\xi_{{k-1}},\xi_{k}}(X_{\tau_{k}}^{x,t},\tau_{k})\biggr )\biggr ].
\end{eqnarray}
Then $u_i:\mathbb R^{N_1}\times[0,T] \to \R$ represents the value function associated with the optimal switching problem on time interval $[t,T]$,
and $u_i(x,t)$ is the optimal expected profit if, at time $t$,
the production is in mode $i$ and the underlying process is at $x$. Under sufficient assumptions it can be proved that the vector $(u_1(x,t),\dots,u_d(x,t))$ satisfies a system of variational inequalities, e.g., see \cite{LNO12}. Using another perspective, the solution to
 the optimization problem can be phrased in the language of reflected backward stochastic differential equations. For these connections, and the application of multi-mode optimal switching problems to economics and mathematics, see \cite{AH09}, \cite{DH09}, \cite{DHP10}, \cite{HM12}, \cite{HT07}, \cite{PVZ09}, \cite{LNO12} and the references therein. More on reflected backward stochastic differential equations in the context of optimal switching problems can be found in \cite{AF12}, \cite{DH09}, \cite{DHP10}, \cite{HT07} and \cite{HZ10}.

\subsection{Optimal switching problems under partial information} In this paper we formulate and consider a multi-mode optimal switching problem under \textit{partial} or \textit{incomplete} information. While
assuming that the running pay-offs in the different modes of production, as well as the cost of switching between modes, depend on
$X=\left(X_s^{x,t}\right)$, with $X=\left(X_s^{x,t}\right)$ as in \eqref{e-SDE}, we also assume that the manager of the production facility can only
observe an auxiliary, and $X$-dependent process, $Y$, based on which the manager can only retrieve partial information of the $N_1$-dimensional stochastic process $X$.  More precisely, we assume that the manager can only observe an $N_2$-dimensional diffusion process
$Y=\left(Y_s^{y,t}\right)$ which solves the system of stochastic differential equations
\begin{eqnarray}\label{e-SDE+}
dY_s^{y,t}&=& h(X_s,s)ds+dU_s,\  t \leq s \leq T,  \notag \\
Y_{t}^{y,t}&=&y.
\end{eqnarray}
Here $(y,t)\in\mathbb R^{N_2}\times[0,T]$ and  $U=\{U_s\}$ is an $m_2$-dimensional Brownian motion, $m_2\leq N_2$, defined on $(\Omega,\mathcal F,\mathbb P)$ and
independent of  $W=\{W_s\}$. $h$ is assumed to be a continuous and bounded function. From here on we let $\F_s^Y=\sigma(Y^{y,t}_u:0\leq u\leq s)$, denote the filtration generated by the observation process $Y$ up to time $s$. Note that in our set-up we have $\F_s^X \not \subset \F_s^Y$, and hence knowledge of the process $Y$ only gives partial information about the process $X$. We emphasize that although the value of the fully observable process $Y$ is known with certainty at time $t$, the value of the process $X$ is not. The observed process $Y$ acts as a source of information for the underlying process $X$. By construction,
in the formulation of an optimal switching problem under partial information we must restrict our strategies, and decisions at time $t$,
to only depend on the information accumulated from $Y$ up to time $t$. Hence, an optimal switching problem under partial information must differ from the standard optimal switching problem in the sense that in the case of partial information, the value of the running payoff functions $\{f_i(X_t,Y_t,t)\}_i$, and the switching costs $\{c_{i,j}(X_t,Y_t,t)\}_{i,j}$, are not known with certainty at time $t$, even though we know $Y_t$. Hence, in this context the production must be run under incomplete information.

Our formulation of an optimal switching problem under partial information is based on ideas and techniques from stochastic filtering theory. Generally speaking, stochastic filtering theory deals with the estimation of an evolving system (``the signal'' $X$) by using observations which only contain partial information about the system (``the observation'' $Y$).  The solution to the stochastic filtering problem is the conditional
distribution of the signal process $X_{t},$ given the $\sigma$-algebra $\F^{Y}_{t}$, and in the context of stochastic filtering theory the goal is to compute the conditional expectations $\E \left[\phi(X_t)|\F^{Y}_{t}\right]$, for suitably chosen test functions $\phi$. In the following the conditional distribution of $X_{t},$ given  $\F^{Y}_{t}$, is denoted by $\pi _{t}$,  i.e.,
\begin{equation}\label{eq:pi}
\E \left[\phi(X_t)|\F^Y_t\right] := \int_{\mathbb{R}^{N_1}}\phi(x)\pi_t(dx) =: \pi_t(\phi).
\end{equation}
Note that the measure-valued (random) process $\pi _{t}$ introduced in \eqref{eq:pi} can be viewed as a stochastic process taking
values in an infinite dimensional space of probability measures over the state space of the signal. Concerning stochastic filtering we refer to \cite{CR11} and \cite{BC09} for a survey of the topic.

 Based on the above we define,  when the production is run using an $\F^Y_t$-adapted strategy $\mu=(\{\tau_k\}_{k\geq 0},\{\xi_k\}_{k\geq 0})$, over a finite horizon $[0,T]$, the total expected
profit up to time $T$ as
\begin{equation}\label{eq1lu}
\tilde J(\mu)=\mathbb E\biggl [\int\limits_0^T\mathbb E\bigl[f_{\mu_s}(X_s,Y_s,s)|\F^{Y}_{s}\bigr]ds -\sum_{0 \leq \tau_k \leq T} \E \bigl[c_{\xi_{k-1},\xi_{k}} (X_{\tau_k},Y_{\tau_k},\tau_k)|\F^{Y}_{\tau_k}\bigr]\biggr ],
\end{equation}
where the  to $\mu$ associated index process $\mu=(\mu_s)$ is defined in the bulk of the paper. Again we are interested in finding an optimal management strategy
$\mu^\ast=(\{\tau_k^\ast\}_{k\geq 0},\{\xi_k^\ast\}_{k\geq 0})$ which maximizes $\tilde J(\mu)$. Let $\mathcal A^Y=\mathcal A^Y[0,T]$ be defined in analogy with $\mathcal A^X=\mathcal A^X[0,T]$ but with $\F^X_t$ replaced by $\F^Y_t$, and let, for given $t\in[0,T]$, $i\in\{1,\dots,d\}$, $\mathcal A_{t,i}^Y\subset \mathcal A^Y$, be the subset of strategies such that $\tau_1 \geq t$ and $\xi_0=i$ a.s. Given $(y,t)\in\R^{N_2}\times[0,T]$, and a measure of finite mass $\gamma$, we let
\begin{align}\label{eq:valuefcnpartial}
v_i(\gamma,y,t)=\sup_{\mu\in \mathcal A^Y_{t,i}}\E &\biggl [\int _t ^T \pi_s \l ( f_{\mu}(\cdot , Y_s,s) \r ) ds
 \notag\\
 &- \sum _{ t \leq \tau_k \leq T} \pi_{\tau_k} \l
 ( c_{\xi_{k-1},\xi_k} (\cdot ,Y_{\tau_{k}},\tau_k) \r ) \, \vline \, \pi_t=\gamma, Y_t =y \biggr ].
\end{align}
Then $v_i:\mathbb R^{N_2}\times[0,T] \to \R$ represents the value function associated with the optimal switching problem under partial information formulated above, on the time interval $[t,T]$,
and $v_i(\gamma,y,t)$ is the optimal expected profit if, at time $t$,
the production is in mode $i$, $Y_t=y$ and the distribution of $X_t$ is given by $\gamma$,  $X_t\sim \gamma$. Note that for a test function $\phi$, $\E \left[\phi(X_t)|\F^{Y}_{t}\right]$ is an $\F_t^Y$-adapted random variable and hence, the problem in \eqref{eq:valuefcnpartial} can be seen as a full information problem with underlying process $Y$. In fact, it is this connection to an optimal switching problem with perfect information that underlies our formulation of the optimal switching problem under partial information. Furthermore, note that if $X$ is an $\{\mathcal F^Y\}$-adapted process, then \eqref{eq:valuefcnpartial} reduces to \eqref{eq1+}, i.e., to the standard optimal switching problem under complete information.

The object of study in this paper is the value function $v_i(\gamma,y,t)$ introduced in \eqref{eq:valuefcnpartial} and  we emphasize and iterate the probabilistic interpretation of the underlying problem in
\eqref{eq:valuefcnpartial}. In \eqref{eq:valuefcnpartial} the manager wishes to maximize $\tilde J(\mu)$ by selecting an optimal  $\mu^\ast$. However, the manager only has access to the
observed process $Y$. The state $X$ is not revealed and can only be partially inferred through its
impact on the drift of $Y$. Thus, $\mu^\ast$ must be based on the information contained solely in $Y$, i.e., $\mu^\ast$ must be $\F^Y_t$-adapted. Hence, the optimal switching problem under partial information considered here gives a model for  the decision making of a manager who is not fully aware of the economical environment he is acting in. As pointed out in \cite{L09}, one interesting
feature here is the interaction between learning and optimization. Namely, the observation process $Y$
plays a dual role as a source of information about the system state $X$, and as a reward ingredient.
Consequently, the manager has to consider the trade-off between further monitoring $Y$ in order to
obtain a more accurate inference of $X$, vis-a-vis making the decision to switch to other modes of production in case the state of the world is unfavorable.
% Compared to the fully observed setting, it is therefore, at least intuitively, natural to expect that partial
%information could postpone the decisions to switch due to the demand for learning.

\subsection{Contribution of the paper} The contribution of this paper is fourfold. Firstly, we are not aware of any papers dealing with optimal switching problems under partial information and we therefore think that our paper represent a substantial contribution to the literature devoted to optimal switching problems and to stochastic optimization problems under partial information. Secondly, we propose  a theoretically sound and entirely simulation-based approach to calculate the value function $v_i(\gamma,y,t) $ in \eqref{eq:valuefcnpartial} when $X$ and $Y$ satisfy the Kalman-Bucy setting of linear stochastic filtering. In particular, we propose
a probabilistic numerical algorithm to approximate $v_i(\gamma,y,t) $ in \eqref{eq:valuefcnpartial} based on dynamic programming and regression Monte Carlo methods. Thirdly, we carry out a rigorous error analysis and prove the convergence of our scheme. Fourthly, we illustrate some of the features of partial information in a computational example. It is known that in the linear Kalman-Bucy setting it is possible to solve the stochastic filtering problem analytically and describe the a posteriori probability distribution $\pi$ explicitly. Although much of the analysis in this paper also holds in the non-linear case, we focus on the, already quite involved, linear setting. In general, numerical schemes for optimal switching problems, already under perfect information, seem to be a less developed area of research and we are only aware of the papers \cite{ACLP12} and \cite{GKP12} where numerical schemes are defined and analyzed rigorously. Our research is influenced by \cite{ACLP12} but our setting is different since we consider an optimal switching problem assuming only partial information.

%and we refer to \cite{LNO13} for a treatment in the non-linear setting.

%Outside of the linear context however, only a few very exceptional examples have explicitly described posterior distributions. Therefore, typically this distribution must be approximated numerically. There is a wide variety of approaches to do this and we refer to \cite{BHL98}, \cite{CLZ95}, \cite{CR11} and \cite{LMR97} and the references therein for some different suggestions. We also refer to Chapters 8 and 9 in \cite{BC09}. The paper \cite{L09} deals with stochastic filtering in the related setting of American options.

\subsection{Organization of the paper} The paper is organized as follows. Section 2 is of preliminary nature and we here state the assumptions on the systems in \eqref{e-SDE}, \eqref{e-SDE+}, the payoff rate in production mode $i$, $f_i$, and
the switching costs $c_{i,j}$, assumptions used throughout the paper. Section 3 is devoted to the general description of the stochastic filtering problem and the linear Kalman-Bucy filter. In Section 4 we prove that the value function $v_i$ in \eqref{eq:valuefcnpartial} satisfies the dynamic programming principle. This is the result on which the numerical scheme, outlined in the subsequent sections, rests. Section 5 gives, step by step, the details of the proposed numerical approximation scheme.
%The value produced by the algorithm is  $ \tilde v_i^{\check\Pi}(t_k,\bar \rho^n,y)$ which is an approximation of the true value $v_i(t,\rho,y)$.
In Section \ref{sec:convanalysis} we perform a rigorous mathematical convergence analysis of the proposed numerical approximation scheme and the main result, Theorem \ref{thm:convergence}, is stated and proved. We emphasize that by proving Theorem \ref{thm:convergence} we are able to establish a rigorous error control
% i.e., a rigorous bound on $|\tilde v_i^{\check\Pi}(t_k,\bar \rho^n,y)-v_i(t,\rho,y)|$,
for the proposed numerical approximation scheme.
%Naturally, compared to \cite{L09}, \cite{ACLP12}, in our case things are more complicated due to the optimal switching problem and the presence of the underlying filtering problem.
Section 7 contains a numerical illustration of our algorithm and the final section, Section 8, is devoted to a summary and conclusions.

\setcounter{equation}{0} \setcounter{theorem}{0}
%%%%%%%%%%%%%%%%%%%%%%%%%%%%%%%%%%%%%%%%%%%%%%%%%%%%%%%%%%%%%%%%%%%%%%%%%%%%%
\section{Preliminaries and Assumptions}\label{sec:prel}
\noindent
We first state the assumptions on the systems \eqref{e-SDE}, \eqref{e-SDE+}, the payoff rate in production mode $i$, $f_i$, and
the switching costs, $c_{i,j}$, which will be used in this paper. We let $\Q = \{1,\dots, d\}$ denote the (finite) set of available states and we let  $\Q^{-i}=\Q\setminus\{i\}$ for $i\in\{1,...,d\}$. As stated, the profit made (per unit time) in state $i$ is given by the function $f_i$. The cost of switching from state $i$ to state $j$ is given by the function $c_{i,j}$. Focusing on the problem in \eqref{eq1lu}, and in particular on the value function in \eqref{eq:valuefcnpartial}, we need to give a precise definition of the strategy process $\mu=\mu_s$ and the notation $f_{\mu_s}$. Indeed, in our context
a strategy $\mu$, over a finite horizon $[0,T]$, corresponds to a sequence $(\{\tau_k\}_{k\geq 0},\{\xi_k\}_{k\geq 0})$, where
$\{\tau_k\}_{k\geq 0}$ is a sequence of $\mathcal{F}^Y$-adapted stopping times and $\{\xi_k\}_{k\geq 0}$ is a sequence of measurable
random variables taking values in $\Q$ and such that $\xi_k$ is  $\mathcal{F}_{\tau_k}^Y$-adapted. Given $\mu=(\{\tau_k\}_{k\geq 0},\{\xi_k\}_{k\geq 0})$ we let
\begin{eqnarray*}%\label{index}
\mu_s=\xi_0\chi_{[0,\tau_0)}(s)+\sum_{k\geq 0}\xi_k\chi_{[\tau_k,\tau_{k+1})}(s)\in \Q,
\end{eqnarray*}
where $\chi_{B}(s)$ is the indicator function for a measurable set $B\subset\mathbb R$,  be the associated index process. In particular, to each strategy
$\mu=(\{\tau_k\}_{k\geq 0},\{\xi_k\}_{k\geq 0})$ there is an associated index process $\mu=(\mu_s)$ and this is the process used in the definition of
$f_{\mu_s}$.

We  denote by $C^k_b(\R^N)$ the space of all real-valued functions
 $g: \R^N \to \R$ such that $g$ and all its partial derivatives up to order $k$ are continuous and bounded on $\R^N$. Given
 $g\in C^k_b(\R^N)$ we let
$$
\|g\|_{k,\infty} = \sum _{|\alpha| \leq k } \sup_{x \in \R^N} |D^\alpha g(x)|.
$$
Similarly, we denote by $C^{2,1}_b(\R^{N}\times[0,T])$ the space of all real-valued functions
 $g: \R^{N}\times[0,T] \to \R$ such that $g$, $Dg$, $D^\alpha g$, $|\alpha|=2$, and $\partial_tg$ are continuous and bounded on $\R^{N}\times[0,T]$. With a slight abuse of notation we will often write $\|g\|_\infty$ instead of $\|g\|_{0,\infty}$. We denote by $\Gamma(\R^N)$ the space of of all positive measures on $\mathbb R^N$ with finite mass. Considering the systems in \eqref{e-SDE}, \eqref{e-SDE+}, we assume that $b:{\R^{N_1}}\times[0,T] \to \R^{N_1}$, $\sigma :{\R^{N_1}}\times[0,T] \to \mathcal M_{N_1,m_1}$ and $h:{\R^{N_1}}\times[0,T] \to \R^{N_2}$ are continuous and bounded functions. Here
$\mathcal M_{N_1,m_1}$ is the set of all $N_1\times m_1$-dimensional real-valued matrices. Furthermore, concerning the regularity of these functions we assume that
\begin{eqnarray}\label{assump1elel}
b_i,\sigma_{i,j},h\in C^{2,1}_b(\R^{N_1}\times[0,T]).
\end{eqnarray}
Clearly \eqref{assump1elel} implies that
\begin{eqnarray}\label{assump1elel+}
|b_i(x,t)-b_i(y,t)|+|\sigma_{i,j}(x,t)-\sigma_{i,j}(y,t)|+|h(x,t)-h(y,t)|\leq A|x-y|,
\end{eqnarray}
for some constant $A$, $1\leq A<\infty$, for all $ i,j$, and whenever $(x,t)\in\mathbb R^{N_1}\times [0,T]$. Here $|x|$ is the standard Euclidean norm of $x\in\mathbb R^{N_1}$. Given \eqref{assump1elel} and \eqref{assump1elel+}, we see, using the standard existence theory for stochastic differential equations, that
there exist unique strong solutions $X_t$ and $Y_t$ to the systems in \eqref{e-SDE} and \eqref{e-SDE+}.  Concerning regularity of the payoff functions
$\{f_i\}_i :\R^{N_1} \times \R^{N_2}\times[0,T]\to \R$ and the switching costs $c_{i,j}: \R^{N_1} \times \R^{N_2}\times[0,T]\to \R$, we assume that
\begin{eqnarray*}%\label{ass:functions1}
f_i, c_{ij} \in C^{2,1}_b(\R^{N_1+N_2}\times[0,T])=C^{2,1}_b(\R^{N_1} \times \R^{N_2}\times[0,T]).
\end{eqnarray*}
For future reference we note, in particular, that
\begin{eqnarray}\label{ass:functions2}
(i) &&|f_i(x,y,t)-f_i(x', y',t')| \leq c_1 \l (|x-x'| + |y-y'|+|t-t'| \r ) ,\notag\\
(ii) &&|c_{i,j}(x,y,t)-c_{i,j}(x', y',t')|\leq c_2\l(|x-x'|  + |y-y'|+|t-t'|\r ),
\end{eqnarray}
for some constants $c_1,c_2$, $1\leq c_1,c_2<\infty$ whenever $(x,y,t),(x', y',t')\in\R^{N_1} \times \R^{N_2}\times[0,T]$. Note that \eqref{ass:functions2} implies that $f_i(x,y,t)$ and $c_{i,j}(x,y,t)$, $i,j \in \Q$, are, for $t$ fixed, Lipschitz continuous w.r.t. $x$, uniformly in $y$, and vice versa. Concerning the switching costs we also impose the following structural assumptions on the functions $\{c_{i,j}\}$,
 \begin{eqnarray}\label{ass:switchingcosts}
 (i)&&c_{i,i}(x,y,t)=0\mbox{ for each $i\in\{1,\dots.,d\}$},\notag\\
 (ii)&&c_{i,j}(x,y,t)\geq\nu\mbox{ for some $\nu>0$, when $i,j\in \Q$, $(x,y,t)\in \R^{N_1} \times \R^{N_2}\times[0,T]$},\notag\\
 (iii)&&\mbox{$c_{i_1,i_2}(x,y,t)+c_{i_2,i_3}(x,y,t)\geq c_{i_1,i_3}(x,y,t)$ for all $(x,y,t)\in\R^{N_1} \times \R^{N_2}\times[0,T]$,}\notag\\
 &&\mbox{and for any sequence of indices $i_1$, $i_2$, $i_3$, $i_j\in\{1,\dots,d\}$ for $j\in\{1,2,3\}$.}
\end{eqnarray}
Note that \eqref{ass:switchingcosts} $(iii)$ states that it is always less expensive to switch directly from state $i$ to state $k$ compared to passing  through an intermediate state $j$. We emphasize that we are able to carry out most of the analysis in the paper assuming only \eqref{assump1elel}--\eqref{ass:switchingcosts}. However, there is one instance where we currently need to impose stronger structural restrictions on 
the functions $\{c_{i,j}\}$ to pull the argument through. Indeed, our final argument relies heavily on the Lipschitz continuity of certain value functions, established in Lemma \ref{lemma:Lipschitzm} and Lemma \ref{lemma:Lipschitzy} below. Currently, to prove these lemmas we need the extra assumption that
 \begin{eqnarray}\label{ass:switchingcosts+}
 c_{i,j}(x,y,t)=c_{i,j}(t)\mbox{ when $i,j\in \Q$, $(x,y,t)\in \R^{N_1} \times \R^{N_2}\times[0,T]$}.
\end{eqnarray}
In particular, we need the switching cost to depend only on $t$ and the sole reason is that we need to be able to estimate terms of the type $A_3$ appearing
in the proof of Lemma \ref{lemma:Lipschitzm} (Lemma \ref{lemma:Lipschitzy}). While we strongly believe that these lemmas remain true without \eqref{ass:switchingcosts+}, we also believe that the proofs in the more general setting require more refined techniques beyond the dynamic programming principle, and that we have to resort to the connection to systems of variational inequalities with interconnected obstacles and reflected backward stochastic differential equations.

%\eqref{ass:switchingcosts} $(iii)$ is also stronger compared to the following so called non-loop condition:
% \begin{eqnarray}\label{assump3+}
%&&\mbox{For any sequence of indices $i_1$,\dots, $i_k$, $i_j\in\{1,\dots,d\}$ for each $j\in\{1,\dots,k\}$,}\notag\\
% &&\mbox{we have $c_{i_1,i_2}(x,y,t)+c_{i_2,i_3}(x,y,t)+\dots+c_{i_{k-1},i_k}(x,y,t)+c_{i_k,i_1}(x,y,t)>0$},\notag\\
% &&\mbox{for all $(x,y,t)\in \R^{N_1} \times \R^{N_2}\times[0,T]$.}
%\end{eqnarray}

\setcounter{equation}{0} \setcounter{theorem}{0}
\section{The filtering problem} \label{sec:KB}
\noindent
As outlined in the introduction, the general goal of the filtering problem is to find the conditional distribution of the signal $X$ given the observation $Y$. In particular, given $\phi\in C^{2}_b(\R^{N_1})$,
\begin{equation*}%\label{eq:pi+}
\pi_t(\phi) = \int_{\mathbb{R}^{N_1}}\phi(x)\pi_t(dx)=\mathbb E\left[\phi(X_t)|\F^Y_t\right],
\end{equation*}
and the aim is to find the (random) measure $\pi _{t}$. Note that $\pi_t$ can be viewed as a stochastic process taking
values in the infinite dimensional space of probability measures over the state space of the signal. Let
\begin{eqnarray*}%\label{assump1el}
a_{i,j}(x,y,t)=\frac 1 2(\sigma(x,y,t)\sigma^\ast(x,y,t))_{i,j},\ i,j\in\{1,\dots,{m_1}\},
\end{eqnarray*}
where $\sigma^\ast$ is the transpose of $\sigma$, and let $\mathcal H$ be the following partial differential operator
\begin{equation*}%\label{opera}
    \mathcal{H}=\sum_{i,j=1}^{m_1} a_{i,j}(x,t)\partial_{x_i x_j}+\sum_{i=1}^{m_1} b_i(x,t)\partial_{x_i}+\partial_t.
\end{equation*}
Using this notation and the assumptions stated in Section \ref{sec:prel}, one can show, e.g., see \cite{BC09}, that the stochastic process $\pi=\{\pi_t:t\geq0\}$ satisfies %\fxnote{Citation?}
\begin{equation}\label{eq.kushner_stratonovich_equation}
d\pi_t(\phi)=\pi_t(\H\phi)dt+\sum_{k=1}^{N_2}[\pi_t(h_k\phi)-\pi_t(h_k)\pi_t(\phi)][dY_t^k-\pi_t(h_k)dt],
\end{equation}
for any $\phi\in C^{2}_b(\R^{N_1})$. Recall that $h=(h_1,\dots,h_{N_2})$ is the function appearing in \eqref{e-SDE+}. The non-linear stochastic PDE in \eqref{eq.kushner_stratonovich_equation} is called the Kushner-Stratonovich equation. Furthermore, it can also be shown, under certain conditions, that the
Kushner-Stratonovich equation has, up to indistinguishability, a pathwise unique solution, e.g., see \cite{BC09}. From here on in we will, to simplify the notation, write
\begin{equation*}%\label{eq:expectationnotation2}
\mathbb E^{\gamma,y,t}\left  [\cdot \right  ] := \mathbb E \left[ \cdot \, \vline \,   \pi_t= \gamma, Y_t=y   \right].
%\tilde {\mathbb  E}^{\gamma,y,t}\left  [\cdot \right  ] := \tilde{\mathbb %E} \left[ \cdot \, \vline \,   \pi_t= \gamma, Y_t=y   \right].
\end{equation*}

\subsection{Kalman-Bucy filter}
It is known that in some particular cases the filtering problem outlined above can be solved explicitly and hence the a posteriori distribution $\pi_t$ is known. In particular, if we assume that the signal $X$ and the observation $Y$ solve linear SDEs, then the solution to the filtering problem can be given explicitly. To be even more specific, assume that the signal $X$ and the observation $Y$ are given by the systems in \eqref{e-SDE}, \eqref{e-SDE+}, with
\begin{equation} \label{eq:KB}
b(x,t)=F_t x, \hspace{0.5cm} \sigma(x,t)=C_t, \hspace{0.5cm} \mbox{and} \hspace{0.5cm} h(x,t)=G_t x,
\end{equation}
 respectively, where $F_t : [0,T] \to \R^{N_1 \times N_1}$, $C_t : [0,T] \to \R^{N_1 \times m_1}$ and $G_t : [0,T] \to \R^{N_2 \times N_1}$ are measurable and locally bounded time-dependent functions. Furthermore, assume that $X_0 \sim \N(m_0,\theta_0)$, where $\N(m_0,\theta_0)$ denotes the $N_1$-dimensional multivariate normal distribution defined by the vector of means $m_0$ and by the covariance matrix $\theta_0$,  is independent of the underlying Brownian motions $W$ and $U$.  Let $m_t:=\E[X_t \, \vline \, \F^Y_t]$ and $\theta _t := \E \l [ \langle X_t - m_t, X_t - m_t \rangle \r ]$ denote the conditional mean and the covariance matrix of  $X_t$, respectively. The following results concerning the filter $\pi_t$ and the processes $m_t$ and $\theta_t$, can be found in, e.g.,  \cite{KB61} or Chapter 6 in \cite{BC09}.
 \begin{theorem} \label{thm:multivariatenormal} Assume \eqref{eq:KB} and that $X_{t_0} \sim \N(m_{t_0},\theta_{t_0})$ for some $t_0\in [0,T]$.
Then the conditional distribution $\pi_t$ of $X_t$, conditional on $\F^Y_t$, is a multivariate normal distribution, $X_t\sim  \N(m_t,\theta_t)$.
\end{theorem}
\begin{theorem}\label{thm:filter} Assume \eqref{eq:KB} and that $X_{t_0} \sim \N(m_{t_0},\theta_{t_0})$ for some $t_0\in [0,T]$.
Then the conditional covariance matrix $\theta_t$ satisfies the deterministic matrix equation
\begin{align} \label{eq:theta}
\frac{d\theta_t}{dt} = F_t\theta_t + \theta_t F^\ast_t - \theta_t G^\ast_t G_t \theta_t + C_tC_t^\ast,\mbox{ for $t\in[t_0,T]$},
\end{align}
with initial condition $\theta_{t_0} = \E \l [ \langle  X_{t_0} - \E [X_{t_0}], X_{t_0} - \E [X_{t_0}] \rangle \r ]$,
and the conditional mean $m_t$ satisfies the stochastic differential equation
\begin{align} \label{eq:theta+}
dm_t= (F_t-\theta_tG_t^\ast G_t)m_t dt + \theta_t G^\ast_t dY_t,\mbox{ for $t\in[t_0,T]$},
\end{align}
with initial condition $m_{t_0}= \E [X_{t_0}]$.
\end{theorem}

For a positive semi-definite matrix $A$, let $A^{1/2}$ denote the unique positive semi-definite matrix $R$ such that $RR^\ast =A$, where, as above, $R^\ast$ denotes the transpose of $R$. Recalling that
the density defining $\N(m_t,\theta_t)$ in $\R^{N_1}$, at $z$, equals 
$$
(2\pi(\det \theta_t)^{1/N_1})^{-N_1/2}\exp(-(z-m_t)^\ast\theta_t^{-1}(z-m_t)/2),
$$ 
we see that the following result follows immediately from  Theorem \ref{thm:multivariatenormal} and Theorem \ref{thm:filter}.
\begin{corollary}\label{cordir}
The distribution $\pi_t$ is fully characterized by $m_t$ and $\theta_t$ and
\begin{align*}
\pi_t(\phi) = \frac{1}{(2\pi)^{N_1/2}} \int _{\R^{N_1}} \phi (m_t + \theta_t^{1/2} z) \exp(- \dfrac{|z|^2}{2}) dz
\end{align*}
for any $\phi \in C_b^2(\mathbb{R}^{N_1})$.
\end{corollary}

Note that the covariance matrix $\theta_t$ is deterministic and depends only on the known quantities $F_t$, $G_t$, $C_t$, see \eqref{eq:KB}, and the distribution $\gamma\sim \N(m_{t_0},\theta_{t_0})$ of the starting point of $X$. Hence, once the initial distribution $\pi_{t_0}$ is given, the covariance matrix $\theta_t$ can be determined for all $t \in [t_0,T]$. Furthermore, in the Kalman-Bucy setting, the measure $\pi_t$ is Gaussian and hence fully characterized by its mean $m_t$ and its covariance matrix $\theta_t$. As a consequence, the value function
to the partial information optimal switching problem, $v_{i}(X_t, Y_t,t)$, can in this setting be seen as a function $v_{i}(m_t, \theta_t, Y_t,t) : \R^{N_1} \times \R^{(N_1\times N_1)} \times \R^{N_2}\times[0,T] \to \R$. 
We will, when $\pi$ is a Gaussian measure with mean $m_t$ and covariance matrix $\theta_t$, write
\begin{equation*}%\label{eq:expectationnotation}
\mathbb E^{m,\theta, y,t}\left  [\cdot \right  ] := \mathbb E \left[ \cdot \, \vline \,   m_t=m, \theta_t=\theta, Y_t=y  \right]\ \text{or}\
\mathbb E^{m,y,t}\left  [\cdot \right  ] := \mathbb E \left[ \cdot
\, \vline \,   m_t= m, Y_t=y\right].
\end{equation*}

\begin{remark}\label{remark:timeshift}
Consider a fixed $t_0\in [0,T]$,  and let $\hat F_s = F_{s+t_0}$, $\hat C_s = C_{s+t_0}$, $\hat G_s = G_{s+t_0}$, whenever $s\in[0,T-t_0]$. Let $\hat f_i(\cdot,\cdot,s) = f_i(\cdot,\cdot,{s+t_0})$ and $\hat c_{i,j}(\cdot,\cdot,s) = c_{i,j}(\cdot,\cdot,s+t_0)$. Let
$X_{t}$, with initial distribution determined by $m_t$ and $\theta_t$, and $Y_{t}$ be given as above for $t\in [t_0,T]$. Furthermore, given $X_{t_0}$ and $Y_{t_0}$, let $\hat X_s$ and $\hat Y_s$
be the unique solutions  to the systems in \eqref{e-SDE} and \eqref{e-SDE+}, with $b$, $\sigma$, $h$, defined as in \eqref{eq:KB} but with
$F,C,G$ replaced by $\hat F,\hat C,\hat G$ and with initial data $\hat X_0=X_{t_0}$ and $\hat Y_0=Y_{t_0}$. In addition, let $\hat m_t, \hat \theta_t$ be defined as in \eqref{eq:theta} and \eqref{eq:theta+}, with $\hat m_0=m_{t_0}$ and $\hat \theta_0=\theta_{t_0}$. Finally, consider the value function
$v_{i}(m_t, \theta_t, Y_t,t)$  and let $\hat v_i (\hat m_t, \hat \theta_t, \hat Y_t,t)$ be the value function of the optimal switching problem on $[0,T-t_0]$, with $ F, G,  C,  f_i, c_{i,j}$ replaced by $\hat F, \hat G, \hat C, \hat f_i, \hat c_{i,j}$. Then
$$v_{i}(m_t, \theta_t, Y_t,t)=\hat v_i (\hat m_{t-t_0}, \hat \theta_{t-t_0}, \hat Y_{t-t_0},{t-t_0})\mbox{ whenever $t\in[t_0,T]$}.$$
In particular,
$$v_{i}(m_{t_0}, \theta_{t_0}, Y_{t_0},{t_0})=\hat v_i (\hat m_{0}, \hat \theta_{0}, \hat Y_{0},{0})$$
and we see that there is no loss of generality to assume that initial observations are made at $t=0$. \end{remark}
\begin{remark}\label{rmk:theta}
As the covariance matrix $\theta_{t}$ solves the deterministic Riccati equation in \eqref{eq:theta}, it is completely determined by the parameters of the model and the covariance matrix of $X_t$ at time $t=0$. Hence, once the initial condition $\theta_0$ is given, $\theta_t$ can be solved deterministically for all $t \in [0,T]$, and consequently viewed as a known parameter. Therefore, we omit the dependence of $\theta_{t}$ in the value function $v_i(m_t,\theta_t,Y_t,t)$ and instead, with a slight abuse of notation, simply write $v_i(m_t,Y_t,t)$.
\end{remark}

\begin{remark}\label{remthetanud}
Although \eqref{eq:theta} is a deterministic ordinary differential equation, it may not be possible to solve it analytically. Therefore, in a general numerical treatment of the problem outlined above one has to use numerical methods to find the covariance matrix $\theta $. The error stemming from the numerical method used to solve \eqref{eq:theta} will have influence on the total error, defined as the absolute value of the difference between the true value function $v_i(m_t,Y_t,t)$ and its numerical approximation derived in this paper. However, as $\theta$ is deterministic, it can be solved off-line and to arbitrary accuracy without effecting the computational efficiency of the main numerical scheme presented in this paper. Therefore, we will throughout this paper consider $\theta$ as exactly known and ignore any error caused by the numerical algorithm used for solving \eqref{eq:theta}.
\end{remark}

\subsection{Connection to the full information optimal switching problem}
As mentioned in the introduction, the problem in \eqref{eq:valuefcnpartial} can interpreted as a full information optimal switching problem with underlying process $Y$. We here expand on this interpretation in the context of Kalman-Bucy filters. Let, using the notation in Remark \ref{rmk:theta},
\begin{eqnarray*}
f_{\mu_s}^z(m,y,t)&=&f_{\mu_s}(m+\theta_t^{1/2} z, y,t),\notag\\
c_{ \xi_{k-1}, \xi_k}^z(m,y,t)&=&c_{\xi_{k-1},\xi_k}(m+\theta_t^{1/2}z, y,t),
\end{eqnarray*}
whenever $z\in \R^{N_1}$, and let
$v_i^z(m,y,t)$ be defined through
\begin{align*}%\label{eq:valuefcnpartialb}
v_i^z(m,y,t)=\sup_{\mu\in \mathcal A^Y_{t,i}}\E^{m,y,t}  \biggl [\int _t ^T  f_{\mu_s}^z(m_s, Y_s,s) ds- \sum _{ t \leq \tau_k \leq T}  c_{\xi_{k-1},\xi_k}^z (m_{\tau_k} ,Y_{\tau_{k}},\tau_k) \biggr ].
\end{align*}
Furthermore, let $\bar f_{\mu}$ and $\bar c_{\xi_{k-1},\xi_k}$ be defined as
\begin{eqnarray*}
\bar f_{\mu}(m,y,t)&=&\frac{1}{(2\pi)^{N_1/2}}\biggl [\int_{\R^{N_1}}f_{\mu_s}^z(m,y,t)\exp(- \dfrac{|z|^2}{2}) d z\biggr],\notag\\
\bar c_{\xi_{k-1},\xi_k}(m,y,t)&=&\frac{1}{(2\pi)^{N_1/2}}\biggl [\int_{\R^{N_1}}c_{\xi_{k-1},\xi_k}^z(m,y,t)\exp(- \dfrac{|z|^2}{2}) d z\biggr].
\end{eqnarray*}
Then, for $z\in \R^{N_1}$ fixed, $v_i^z(m,y,t)$ is a solution to an optimal switching problem with perfect information. Using the above notation, we see that
\begin{align}\label{eq:valuefcnpartiala}
v_i(m,y,t)=\sup_{\mu\in \mathcal A^Y_{t,i}}\E^{m,y,t} \biggl [\int _t ^T  \bar f_{\mu_s}(m_s, Y_s,s) ds-\sum _{ t \leq \tau_k \leq T}  \bar c_{\xi_{k-1},\xi_k}(m_{\tau_k} ,Y_{\tau_{k}},\tau_k) \biggr]
\end{align}
and that the upper bound
\begin{align}%\label{eq:valuefcnpartialc}
v_i(m,y,t)\leq \frac{1}{(2\pi)^{N_1/2}}\biggl [\int_{\R^{N_1}}v_i^z(m,y,t)\exp(- \dfrac{|z|^2}{2}) d z\biggr ] \notag
\end{align}
holds. Moreover, based on \eqref{eq:valuefcnpartiala} we see that also $v_i(m,y,t)$ is a solution to an optimal switching problem with perfect information, with payoff rate in production mode $i$, at time $s$, defined by $\bar f_{i}(m_s,y_s,s)$,  and with switching cost, for switching from mode $i$ to mode $j$ at time $s$, defined by $\bar c_{i,j}(m_s,y_s,s)$.

\setcounter{equation}{0} \setcounter{theorem}{0}
\section{The dynamic programming principle}
\noindent
In this section we prove that the value function $v_i$ associated to our
problem satisfies the dynamic programming principle (DPP). This is the
result
on which the numerical scheme outlined in the next section rests. It should be noted that the dynamic programming principle holds for general systems as in \eqref{e-SDE} and \eqref{e-SDE+}, systems which are not necessarily linear.

\begin{theorem} \label{thm:DPP} Let $t\in [0,T]$ and let $v_i(\gamma,y,t)$ be defined as in \eqref{eq:valuefcnpartial}. Then
\begin{align*}
v_i(\gamma,y,t)=\sup_{\mu\in \mathcal A^Y_{t,i}} {\E}^{\gamma,y,t} &\biggl [\int _t ^\tau
\pi_s \left ( f_{\mu_s}(\cdot , Y_s,s) \right ) ds - \sum _{ t \leq \tau_n \leq \tau}
\pi_{\tau_n} \left
 ( c_{\xi_{n-1},\xi_n} (\cdot ,Y_{\tau_n},\tau_n) \right )\\
 &+v_{\xi_\tau}(\pi_\tau,Y_\tau,\tau)\biggr]
 \end{align*}
 for all $\mathcal F^Y$-adapted stopping times $\tau$, $t\leq\tau\leq T$.
\end{theorem}
\begin{proof}
%\textbf{Measurability issues???}
Let $\{ Y_s^{y,t}\}_{s \geq t}$ and $\{\pi_s^{\gamma,t}\}_{s \geq t}$ be the unique solutions to the systems in \eqref{e-SDE+}, \eqref{eq.kushner_stratonovich_equation}, with initial conditions $Y_t=y$ and $\pi_t=\gamma$, respectively. Note that these processes, as well as $X$, are Markov processes. Hence, using the strong Markov property of $Y$ and $\pi$ we have that
\begin{equation}\label{eq:markovproperty}
Y^{y,t}_s=Y_s^{\tau, Y^{y,t}_\tau } \hspace{1cm} \mbox{and} \hspace{1cm}\pi^{\gamma,t}_s = \pi_s^{\tau, \pi_\tau^{\gamma,t}},
\end{equation} for any $\mathcal F^Y$-adapted stopping time $\tau \in [0,T]$ and for all $s$ such that $t\leq \tau\leq s$. Let
$$
J(\gamma,y,t,\mu) = \E^{\gamma,y,t}  \l [\int _t ^T \pi_s \l ( f_{\mu_s}(\cdot , Y_s,s) \r ) ds - \sum _{ t \leq \tau_n \leq T} \pi_{\tau_n} \l  ( c_{\xi_{n-1},\xi_n} (\cdot ,Y_{\tau_n},\tau_n) \r ) \r]
$$
for  $\mu \in \A_{t,i}^Y$. Then,
\begin{equation*}%\label{eq:altvaluefcn}
v_i(\gamma,y,t)= \sup_{\mu \in \A_{t,i}^Y} J(\gamma,y,t,\mu).
\end{equation*}
Next, using \eqref{eq:markovproperty} and the law of iterated conditional expectations we see that
\begin{align}
J(\gamma,y,t,\mu)&= \E^{\gamma,y,t} \biggl [\int _t ^\tau \pi_s \l ( f_{\mu_s}(\cdot , Y_s,s) \r ) ds- \sum _{ t \leq \tau_n \leq \tau} \pi_{\tau_n} \l  ( c_{\xi_{n-1},\xi_n} (\cdot ,Y_{\tau_n},\tau_n) \r )\notag\\
 &\qquad\qquad
+ J(\tau, \pi_\tau^{\gamma,t}, Y_\tau ^{y,t},\mu) \biggr ] \notag \\
&\leq \E^{\gamma,y,t} \biggl [\int _t ^\tau \pi_s \l ( f_{\mu_s}(\cdot , Y_s,s) \r ) ds - \sum _{ t \leq \tau_n \leq \tau} \pi_{\tau_n} \l  ( c_{\xi_{n-1},\xi_n} (\cdot ,Y_{\tau_n},\tau_n) \r )\notag\\
 &\qquad\qquad
+  v_{\mu_\tau}( \pi_\tau^{\gamma,t},Y_\tau^{y,t}, \tau)  \biggr ] \notag
\end{align}
for any $\mathcal F^Y$-adapted stopping time $\tau \in [t,T]$ and any strategy $\mu \in \A^Y_{t,i}$. In particular, since $\mu$ is arbitrary in this deduction we see that
\begin{align}\label{eq:ineq1}
v_i(\gamma,y,t)&=\sup _{\mu \in A_{t,i}^Y}J(\gamma,y,t,\mu)\notag\\
&\leq\sup _{\mu \in A_{t,i}^Y} \E^{t,\gamma,y} \biggl [\int _t ^\tau \pi_s \l ( f_{\mu_s}(\cdot , Y_s,s) \r ) ds - \sum _{ t \leq \tau_n \leq \tau} \pi_{\tau_n} \l  ( c_{\xi_{n-1},\xi_n} (\cdot ,Y_{\tau_n},\tau_n) \r )\notag\\
&\qquad\qquad\qquad +  v_{\mu_\tau}( \pi_\tau^{\gamma,t},Y_\tau^{y,t},\tau)  \biggr ],
\end{align}
for any $\mathcal F^Y$-adapted stopping time $\tau \in [t,T]$. To complete the proof it remains to prove the opposite inequality, i.e., to prove that
\begin{align} \label{eq:ineq2}
v_i(\gamma,y,t)&\geq \sup _{\mu \in A_{t,i}^Y} \E^{t,\gamma,y} \biggl [\int _t ^\tau \pi_s \l ( f_{\mu_s}(\cdot , Y_s,s) \r ) ds - \sum _{ t \leq \tau_n \leq \tau} \pi_{\tau_n} \l  ( c_{\xi_{n-1},\xi_n} (\cdot ,Y_{\tau_n},\tau_n) \r )\notag\\
&\qquad\qquad\qquad +  v_{\mu_\tau}( \pi_\tau^{\gamma,t},Y_\tau^{y,t},\tau)  \biggr ].
\end{align}
Consider $(\gamma,y,t)$ and let $\mu \in \A^Y_{t,i}$ and $\tau \in [t,T]$, be a fixed strategy and a fixed $\mathcal F^Y$-adapted stopping time,
respectively. Recall that all stochastic processes are defined on the probability space $(\Omega,\mathcal F,\mathbb P)$. By
the definition of $v_i$ there exists, for any $\e >0$ and for any $\omega\in\Omega$, $\mu^\epsilon(\omega) \in \A^Y_{\tau(\omega),\mu_{\tau(\omega)}}$, such
that
\begin{equation} \label{eq:e-optimalityuu}
v_i(\pi^{t,\gamma}_{\tau(\omega)},Y^{y,t}_{\tau(\omega)},\tau(\omega)) -   \e \leq J(\pi^{t,\gamma}_{\tau(\omega)},Y^{y,t}_{\tau(\omega)},\tau(\omega),\mu^\e(\omega)).
\end{equation} Given
$\mu$, $\tau$, $\mu^\epsilon$, we define, for all $\omega\in\Omega$,
\begin{equation*}
\hat \mu_s(\omega) =
\begin{cases}
\mu_s(\omega) \hspace{1cm} \mbox {for $s \in [t, \tau(\omega)]$} \\
\mu^\e_s(\omega) \hspace{1cm} \mbox {for $s \in [\tau(\omega), T]$}.
\end{cases}
\end{equation*}
Then $\hat \mu \in\A_{t,i}^Y$ and, again using the law of iterated conditional expectations, we obtain that
\begin{align}
v_i(\gamma,y,t) &\geq J(\gamma,y,t,\hat \mu) \notag\\
 &=\E^{\gamma,y,t} \biggl [\int _t ^\tau \pi_s \l ( f_{\mu_s}(\cdot , Y_s,s) \r ) ds - \sum _{ t \leq \tau_n \leq \tau} \pi_{\tau_n} \l  ( c_{\xi_{n-1},\xi_n} (\cdot ,Y_{\tau_n},\tau_n) \r )\notag\\
 &\qquad\qquad  +  J(\pi_\tau^{\gamma,t},Y_\tau^{y,t},  \tau, \mu^\e)  \biggr ]. \notag
\end{align}
Finally, using \eqref{eq:e-optimalityuu} and the above display we deduce that
\begin{align}\label{eq:e-optimalityuu+}
v_i(\gamma,y,t)
&\geq \E^{\gamma,y,t} \biggl [\int _t ^\tau \pi_s \l ( f_{\mu_s}(\cdot , Y_s,s) \r ) ds - \sum _{ t \leq \tau_n \leq \tau} \pi_{\tau_n} \l  ( c_{\xi_{n-1},\xi_n} (\cdot ,Y_{\tau_n},\tau_n) \r )\notag\\
&\qquad\qquad  +  v_{\mu_{\tau}}(\pi_\tau^{\gamma,t},Y_\tau^{y,t},\tau)  \biggr ] -\e.
\end{align}
Since $\mu$, $\tau$ and $\e$ are arbitrary in the above argument we see that \eqref{eq:e-optimalityuu+} implies \eqref{eq:ineq2}. Combining \eqref{eq:ineq1} and \eqref{eq:ineq2} Theorem \ref{thm:DPP} follows. We note that the proof of the DPP here outlined follows the usual lines and perhaps a few additional statements concerning measurability issues could have been included. However, we here omit further details and refer to the vast literature on dynamic programming for exhaustive proofs of similar statements. \end{proof}

\begin{remark}\label{discDPP}
In this section the dynamic programming principle is proven with the assumption that $\{ Y_s^{y,t}\}$ and $\{\pi_s^{\gamma,t}\}$ are continuous in time. However, the approximation scheme introduced in the following section is based on (Euler) discretized versions of these processes. We here just note that
the proof above can be adjusted to also yield the dynamic programming principle in the context of the discretized processes.
\end{remark}

%\textbf{I think there are possibly involved measurability issues here... %How do we know v is measureable? Is the constructed strategy admissible?}
%\begin{remark}
%Note that we always have $v_i(\cdot,\cdot,T) =0$. Based on this terminal value
%and Theorem \ref{thm:DPP} we can construct a recursive scheme for the numerical
%computation of the value function. Indeed, following the steps outlined in Section
%\ref{sec:recipe}, we are able to obtain a complete and convergent numerical approximation to the value function $v_i$.
%\end{remark}

\setcounter{equation}{0} \setcounter{theorem}{0}
\section{The numerical approximation scheme} % \label{sec:recipe}
\noindent
In this section we introduce a simulation based numerical scheme for determining  $v_i(\gamma,y,t)$ as in \eqref{eq:valuefcnpartial} and give a step by step presentation of the approximations defining the scheme. Recall that $v_i(\gamma,y,t)$ is the optimal expected payoff, starting from state $i$ at time $t$, with initial conditions $X_t \sim\gamma$ and $Y_t=y$.  We will from now on assume the dynamics of $X$ and $Y$ are given by \eqref{e-SDE} and \eqref{e-SDE+}, respectively, with assumption \eqref{eq:KB} in effect. Consequently, the results of Section \ref{sec:KB} are applicable. Based on Remark \ref{rmk:theta} we in the following write $v_i(\gamma,y,t)=v_i(m, y, t)$. Likewise, we will write $\E^{m,y,t}[\cdot]$ instead of $\E^{\gamma,y,t}[\cdot]$. Furthermore, we can and will, w.l.o.g., assume that the initial distribution $\gamma$ is given at time $t=0$ and hence that the value function $v_{i}$, at time $t$, is a function of the conditional mean $m_t$ (and the deterministic $\theta_t$), based on the observation $X_0 \sim \gamma =\N(m_0,\theta_0)$, see Remark \ref{remark:timeshift}. In other words, we assume that the distribution of $X$ is given at time $t=0$ and the task of the controller is to run the facility using updated beliefs of the conditional mean of the signal, conditional upon the information carried by the observation $Y$.

For the convenience of readers, we in this section list the steps of the proposed numerical scheme and the associated notation.  By Theorem \ref{thm:multivariatenormal} the a posteriori distribution $\pi$ is a Gaussian measure, and in the following we denote by $\pi_t^{m_t}$ the Gaussian measure with mean $m_t$ and covariance matrix $\theta_t$. We emphasize that the outcome of the numerical scheme to be outlined, is an approximation of
\begin{eqnarray}\label{prob}
&&\mbox{$v_i(m,y,0)=v_i(m,\theta,y,0)$ for $(m,\theta,y)\in\mathbb R^{N_1}\times \mathbb R^{N_1\times N_1} \times \mathbb R^{N_2}$}\notag\\
&&\mbox{given and fixed}.
\end{eqnarray}
Based on \eqref{prob} we emphasize that we consider the systems in \eqref{e-SDE} and \eqref{e-SDE+} with initial data at $t=0$, assuming the additional structure in \eqref{eq:KB}. In particular, when considering the systems in \eqref{eq:theta} and \eqref{eq:theta+}, for the calculation in \eqref{prob} and with data  at $t=0$, the initial conditions boil down,  all in all,  to the initial condition $(m,\theta,y)$, at $t=0$, for $(m_t,\theta_t,Y_t)$.

Given $T>0$ fixed, and a large positive integer, $N$, we let $\delta=T/N$, $t_k=k\delta,\ k=0,1,\ldots,N$. We let $\Pi=\Pi^\delta$ denote the naturally defined partition of the interval $[0,T]$ based on $\{t_k\}$, i.e., $\Pi=\Pi^\delta=\{0=t_0<t_1<\cdots<t_N=T\}$. Throughout the paper any discretization of time will be identical to $\Pi=\Pi^\delta$.  The following steps constitute our numerical approximation scheme in the context of \eqref{prob} but starting at $t_k$.
%$n$ will denote the number of particle used in the particle filter.
\begin{enumerate}
\item Step 1 -- Bermudan approximation. We restrict the manager to be allowed to switch only at the time points
$\{t_k\}_{k=0}^N$. This results in a Bermudan approximation, $v_i^\Pi(m,y,t_k)$, of $v_i(m,y,t_k)$.
\item Step 2 -- Time discretization and Euler discretization of $m_t$ and $Y_t$. $m_t$ and $Y_t$ are replaced by corresponding discrete versions, also starting at $(m,y)$ at $t=t_k$, based on the Euler scheme and the partition $\Pi=\Pi^\delta$. This results in an approximation,
$v_i^{\check\Pi}(m,y,t_k)$, of $v_i^\Pi(m,y,t_k)$.
\item Step 3 -- Space localization. The processes $m_t$ and $Y_t$ are replaced by versions which are constrained to a bounded convex set $D^\e$. This gives an approximation $v_i^{\check\Pi,\e}(m,y,t_k)$, of $v_i^{\check \Pi}(m,y,t_k)$.
\item Step 4 -- Representation of conditional expectation using true regression. To calculate $v_i^{\check\Pi,\e}(m,y,t_k)$ we use a regression type technique, replacing the future values by a (true) regression. This results in an approximation, $\hat v_i^{\check\Pi,\e}(m,y,t_k)$, of $v_i^{\check\Pi,\e}(m,y,t_k)$.
\item Step 5 -- Replacing the true regression by a sample mean. To calculate $\hat v_i^{\check\Pi,\e}(m,y,t_k)$ we replace
the coefficients in the true regression by their corresponding sample means. This results in an approximation, $\tilde v_i^{\check\Pi,\e}(m,y,t_k)$, of
$\hat v_i^{\check\Pi,\e}(m,y,t_k)$.
\end{enumerate}
The final value produced by the algorithm is  $ \tilde v_i^{\check\Pi,\e}(m,y,t_k)$ and this is an approximation of the true value $v_i(m,y,t_k)$. In the remaining part of this section we will discuss Step~1 -- Step~5 in more detail. The rigorous error analysis is postponed to Section \ref{sec:convanalysis}.

\subsection{Step 1 -- Bermudan approximation}
%\label{subsec:timediscretization}
Let $\A^{Y,\Pi}_{t,i}$ be the set of strategies $$\mu=(\{\tau_n\}_{n\geq 0},\{\xi_n\}_{n\geq 0}) \in \A^Y_{t,i}$$ such that $\tau_n \in \Pi \cap [t,T]$ for all $n$. Based on $\A^{Y,\Pi}_{t,i}$ we let
\[
\begin{split} %\label{eq:bermudan}
v_i^{\Pi}(m_{t_k},Y_{t_k}, t_k)= \sup _{\mu \in \A^{Y,\Pi}_{t_k,i}} \E^ {m_{t_k},Y_{t_k}, t_k} &\biggl [\int_{t_k}^T \pi_s ^{m_{s}}\l ( f_{\mu_{s}}(\cdot, Y_s,s) \r) ds\notag\\
& - \sum _{t_k \leq \tau_n \leq T} \pi_{\tau_{n}}^{m_{\tau_{n}}} \l (c_{\xi_{n-1},\xi_{{n}}}( \cdot, Y_{\tau_{n}},\tau_{n})\r )\biggr ]
\end{split}
\]
and we refer to $v_i^{\Pi}(m_{t_k},Y_{t_k},t_k)$ as the value function of the Bermudan version of our optimal switching problem under partial information. The difference $|v_{i}(m,y,t_k)-v_i^{\Pi}(m,y,t_k)|$ is quantified in Proposition \ref{lemma:bermudan}. 
\\

\noindent
For future reference we here also introduce what we call the {\it Bermudan strategy}.
\begin{definition}\label{def:approximatingstrategy} Let $\mu=(\{\tau_n\}_{n\geq 0},\{\xi_n\}_{n\geq 0}) \in \A^Y_{t,i}$ and let $\Pi=\Pi^\delta$ be given. Let $\tilde \mu = (\{\tilde \tau_n\}_{n\geq 0},\{\tilde \xi_n\}_{n\geq 0})$ 
be the strategy in $\A^{Y,\Pi}_{t,i}$ defined by
$$\tilde \tau_n = \min \{t \in \Pi:  t \geq \tau_n\}, \quad \tilde \xi_n = \xi_n.$$
Then, $\tilde \mu$ is the {\it Bermudan strategy} associated to $\mu$ and $\Pi$.
\end{definition}

\subsection{Step 2 -- Time discretization and Euler discretization of $m$ and $Y$.} \label{subsec:eulerdisc}
Given the continuous time $t$ we let $\check t = \max \{s \in \Pi : s \leq t \}$. In this step we replace the continuous processes $m_t$ and $Y_t$ with their corresponding discrete Euler approximations. To be specific, we first calculate (pathwise) the Euler approximation of the signal $X$, denoted by $\bar X$ and given by the dynamics
\begin{align}%\label{eu1}
d \bar X_{t} &=    F_{ \check t} \bar X_{\check t} dt + C_{\check t} dW_{\check t}. \notag
\end{align}
Based on $\bar X$ we then introduce the discrete processes $\bar Y$ and $\bar m$, Euler approximations of $Y$ and $m$, respectively, and given by
\begin{align}%\label{eu2}
d \bar Y_{t} &=  G_{\check t}(\bar X_{\check t}) dt + dU_{\check t}, \notag
\end{align}
and
\begin{align}%\label{eu3}
d \bar m_{\check t} &= (F_{\check t} - \theta_{\check t} G^\ast_{\check t} G_{\check t}) \bar m_{\check t} dt + \theta_{\check t } G^\ast_{\check t} d\bar Y_{\check t}. \notag
\end{align}
Based on this we have
\begin{align}
\pi_{\check t}^{\bar m_{\check t}}(\phi) = \frac{1}{(2\pi)^{N_1/2}} \int _{\R^{N_1}} \phi (\bar m_{\check t} + \theta_{\check t}^{1/2} z) \exp(- \dfrac{|z|^2}{2}) dz \notag
\end{align}
for any $\phi \in C_b^2(\mathbb{R}^{N_1})$. Recall that we consider $\theta$ as completely known,  see Remark \ref{remthetanud}, and hence $\theta$ is not subject to discretization. Based on the above processes we let
\begin{align}\label{eq:discretizedvaluefcn}
v_i^{\check\Pi}(\bar m_{t_k},\bar Y_{t_k},t_k)
=&\sup _{\mu \in \A^{Y,\Pi}_{t_k,i}} \E^ {m_{t_k},Y_{t_k},t_k} \biggl [ \int _{t_k}^T\pi_{\check s}^{\bar m _{\check s}} \l ( f_{\mu_{{\check s}}}(\cdot, \bar Y_{\check s},\check s) \r ) ds\notag\\
  &\quad\quad\quad\quad\quad\quad\quad\quad- \sum _{t_k \leq \tau_n \leq T}  \pi_{\tau_{n}}^{ \bar m_{\tau_{n}}} \l (c_{\xi_{n-1},\xi_{{n}}}(\cdot,\bar Y_{\tau_{n}}, \tau_{n}) \r)\biggr] \notag\\
=&\sup _{\mu \in \A^{Y,\Pi}_{t_k,i}} \E^ {m_{t_k},Y_{t_k},t_k} \biggl [ \delta \sum_{\ell=k}^N  \pi_{t_\ell}^{\bar m_{t_{\ell}}} \l ( f_{\mu_{t_\ell}}(\cdot, \bar Y_{t_\ell},t_\ell) \r ) \notag \\
 &\quad\quad\quad\quad\quad\quad\quad\quad- \sum _{t_k \leq \tau_n \leq T} \pi_{\tau_{n}}^{ \bar m_{\tau_{n}}} \l (c_{\xi_{n-1},\xi_{{n}}}(\cdot,\bar Y_{\tau_{n}}, \tau_{n}) \r)\biggr].
\end{align}
$v_i^{\check\Pi}$ is the value function based on the discretized time $\check t$ and the Euler approximations $\bar m_t$ and $\bar Y_t$. The difference between $v_i^{\Pi}(m,y,t_k)$ and $v_i^{\check\Pi}(m,y,t_k)$ is quantified in Proposition \ref{lemma:discretization}. 
\\

\noindent
In the following
we will, in an attempt to slightly ease the notation, omit the bar indicating Euler discretization of $m$ and $Y$ and simply write
$v_i^{\check\Pi}(m_{t_k},Y_{t_k},t_k)$ instead of $v_i^{\check\Pi}(\bar m_{t_k},\bar Y_{t_k},t_k)$ when we believe there is no risk of confusion. In particular, we let the very notation $v_i^{\check\Pi}$ also symbolize that $m$ and $Y$ are discretized as above.

\subsection{Step 3 -- Space localization}
A localization in space will be necessary when estimating the errors induced by Step $4$ and Step $5$ below and we here describe this localization. In particular, to be able to work in bounded space time domains we consider, for a fixed parameter $\e \geq 0$,  (time dependent) domains $\{D^\e_{t}\}_{t\in[0,T]}\subset \R^{N_1} \times \R^{N_2}$ such that $D^\e_{t}$ is a convex domain for all $t\in[0,T]$.  We assume that there exists a constant $C_t^\epsilon$, depending only on $t$ and $\e$, such that $D_t^\epsilon \subset Q_{[-C_t^\epsilon,C_t^\epsilon]}$, where $Q_{[-C_t^\epsilon,C_t^\epsilon]}$ is the hypercube in $\R^{N_1} \times \R^{N_2}$ with sides of length $C_t^\epsilon$. Let $Z^{\e}_t$ be the projection of a generic stochastic process $Z_t$ onto the domain $D_t^\e$, i.e., $Z^{\e}_t=Z_t$ if $Z_t \in \overline{D_t^\e}$, and $Z^{\e}_t$ is the, by convexity of $D^\e_{t}$, naturally defined unique projection of $Z_t$ onto $D_t^\e$, along the normal direction, otherwise. Following \cite{ACLP12}, we assume that $D_t^\e$ can be chosen such that
\begin{equation}\label{eq:projection}
\E \l[ |Z_t - Z_t^{\e}| \r] \leq \e\mbox{ for all $t \in [0,T]$}.
\end{equation}
The space $\{D^\e_{t}\}_{t\in[0,T]}$ can be seen as the domain in which the $(N_1+N_2)$-dimensional process $(m^\e_t,Y^\e_t)$ lives. Roughly speaking, condition \eqref{eq:projection} states that most of the time the $(N_1+N_2)$-dimensional process $(m_t,Y_t)$ will be found inside this domain. As mentioned, this localization in space will be necessary when estimating the errors induced by Step $4$ and Step $5$. In particular, based on \eqref{eq:projection} we will in these steps be able to reuse results on the full information optimal switching problem developed in \cite{ACLP12}. We therefore refer to \cite{ACLP12} for more details on the assumption in \eqref{eq:projection} and a constructive example. The construction above stresses generalities but, although not necessary, we will in the follwoing assume, to be consistent and to minimize notation, that
\begin{eqnarray}\label{dom}
&&\mbox{$D^\e_{t}=D^\e$ for all $t\in[0,T]$, for some $D^\e\subset \R^{N_1} \times \R^{N_2}$, and that}\notag\\
&&\mbox{$D^\e= Q^m_{[-C_m^\e,C_m^\e]}\times Q^Y_{[-C_Y^\e,C_Y^\e]}$ for some constants $C_m^\e$ and $C_Y^\e$.}
\end{eqnarray}
In other words, $D^\e$ is assumed to be a time-independent Cartesian product of hypercubes.
The result of this step is that  $\bar m_t$ and $\bar Y_t$ are replaced by their corresponding
projected versions, $\bar m_t^\epsilon$ and $\bar Y_t^\epsilon$, respectively. We let $v_i^{\check\Pi,\e}(m_{t_k},Y_{t_k},t_k)$ denote the associated value function when $\bar m_t$ and $\bar Y_t$ are replaced by $\bar m_t^\epsilon$ and $\bar Y_t^\epsilon$. In particular, when writing $v_i^{\check\Pi,\e}(m_{t_k},Y_{t_k},t_k)$, $\e$ also indicates that the underlying dynamics is that of $\bar m_t^\epsilon$ and $\bar Y_t^\epsilon$. The error introduced by considering $\bar m_t^\epsilon$ and $\bar Y_t^\epsilon$, i.e., the difference between $v_i^{\check \Pi}(m,y,t_k)$ and $v_i^{\check\Pi,\e}(m,y,t_k)$, is quantified in Proposition \ref{lemma:projected}.
%To ease notation slightly, we will from here on in most cases omit the $\e$ %indicating projected versions of the underlying.

\subsection{Step 4 -- Representation of conditional expectation using true regression}
Note that for the discretized Bermudan version of the optimal switching problem, the value function \eqref{eq:discretizedvaluefcn} can, for any $t_k \in \Pi^\delta$, be simplified to read
\begin{align}\label{eq:discretizedvaluefcn+}
v_i^{\check\Pi,\e}(m_{t_k},Y_{t_k},t_k) =\sup _{\mu \in \A^{Y,\Pi}_{t_k,i}} \E^{m_{t_k},Y_{t_k},t_k} &\biggl [ \delta\sum_{n=k}^{N}  \pi^{m_{t_{n}}}_{t_n}(f_{\mu_{t_n}}(\cdot, Y_{t_n},t_n)) \notag\\
 &- \sum _{ t_k \leq \tau_n \leq T}  \pi^{m_{\tau_{n}}}_{\tau_n}(c_{\xi_{n-1},\xi_{n}}(\cdot, Y_{\tau_{n}},\tau_{n}))\biggr ]  ,
\end{align}
where on the right hand side, consequently, $m_{t_n}=\bar m_{t_n}^\epsilon$,  $Y_{t_n}=\bar Y_{t_n}^\epsilon$. As a result, the DPP for \eqref{eq:discretizedvaluefcn+} in the discretized Bermudan setting, see Remark \ref{discDPP}, reduces to
\begin{align} %\label{eq:DPPPi}
v_i^{\check \Pi,\e}(m_{t_k},Y_{t_k},t_k)= \max &\bigg \{ \delta \pi^{m_{t_{k}}}_{t_k}(f_{i}(\cdot, Y_{t_k},t_k))+ \E^{m_{t_k},Y_{t_k},t_k} \l [ v_i^{\check\Pi,\e}(m_{t_{k+1}},Y_{t_{k+1}},t_{k+1}) \r ], \notag  \\
&\qquad\max _{j \in \Q^{-i}}  \l \{ v_j^{\check\Pi,\e}(m_{t_k},Y_{t_k},t_k) - \pi^{m_{t_{k}}}_{t_k}(c_{i,j}(\cdot, Y_{t_k},t_k)) \r \}   \bigg \}. \notag
\end{align}
Taking \eqref{ass:switchingcosts} $(i)$ and $(iii)$ into account this can be further simplified to
\begin{align}\label{eq:DPPPi}
v_i^{\check \Pi,\e}(m_{t_k},Y_{t_k},t_k)=  \max_{j \in \Q}&\bigg \{\delta \pi^{m_{t_{k}}}_{t_k}(f_{j}(\cdot, Y_{t_k},t_k))+ \E^{m_{t_k},Y_{t_k},t_k} \l [  v_j^{\check \Pi,\e}(m_{t_{k+1}},Y_{t_{k+1}},t_{k+1}) \r ]\notag\\
&-  \pi^{m_{t_{k}}}_{t_k}(c_{i,j}(\cdot, Y_{t_k},t_k)) \bigg \}.
\end{align}
Based on \eqref{eq:DPPPi} the recursive scheme based on the DPP becomes
\begin{align} \label{eq:scheme}
v_i^{\check \Pi,\e}(m_T,Y_T,T)&=0 \notag \\
v_i^{\check \Pi,\e}(m_{t_k},Y_{t_k},t_k) &= \max_{j \in \Q} \bigg \{ \delta\pi^{m_{t_{k}}}_{t_k}(f_{j}(\cdot, Y_{t_k},t_k)) + \E^{m_{t_k},Y_{t_k},t_k} \l [  v_j^{\check \Pi,\e}(m_{t_{k+1}},Y_{t_{k+1}},t_{k+1}) \r ]\notag\\
&\qquad\qquad-  \pi^{m_{t_{k}}}_{t_k}(c_{i,j}(\cdot, Y_{t_k},t_k))  \bigg \}.
\end{align}
An important feature of the scheme in \eqref{eq:scheme} is that $v_i^{\check \Pi,\e}(m_{t_k},Y_{t_k},t_k)$ depends explicitly on the value functions at time $t_{k+1}$, $v_{j}^{\check\Pi,\e}(m_{t_{k+1}},Y_{t_{k+1}},t_{k+1})$. However, these functions are unknown at time $t_k$. Furthermore, the optimal strategy at time $t_k$ also depends directly on this future value. Indeed, at time $t_k$ it is optimal to switch from state $i$ to $j$ if the difference between the expected future value retrieved from being in mode $j$ and the switching cost $c_{i,j}$, is greater than the expected profit made from staying in state $i$. More precisely, at time $t_k$ it is optimal to switch from state $i$ to $j$ if
\begin{align*}
&\ \max _{j \in \Q^{-i}} \biggl \{\delta\pi^{m_{t_{k}}}_{t_k}(f_{j}(\cdot, Y_{t_k},t_k))+ \E^{m_{t_k},Y_{t_k},t_k} \l [  v_j^{\check \Pi}(m_{t_{k+1}},Y_{t_{k+1}},t_{k+1}) \r ]\notag\\
&\qquad\quad-\pi^{m_{t_{k}}}_{t_k}(c_{i,j}(\cdot, Y_{t_k},t_k))\biggr\} \notag \\
&>\delta\pi^{m_{t_{k}}}_{t_k}(f_{i}(\cdot, Y_{t_k},t_k)) + \E^{m_{t_k},Y_{t_k},t_k} \l [  v_i^{\check \Pi}(m_{t_{k+1}},Y_{t_{k+1}},t_{k+1}) \r ].
\end{align*}
If this inequality holds with $>$ replaced by $\leq $ it is optimal to stay in state $i$. To construct an $\F^Y_t$-adapted strategy it is hence necessary to estimate the future expected value at time $t_{k+1}$, based on the information available at time $t_k$, i.e., to estimate
\begin{eqnarray}\label{reg1}
\E^{m_{t_k},Y_{t_k},t_k}\biggl [ v_i^{\check \Pi,\e}(m_{t_{k+1}},Y_{t_{k+1}},t_{k+1})\biggr ].
\end{eqnarray}
%\fxnote{Should this line really be here?} No, it can be removed.
%It is necessary to assume, for a given function $\varphi$, that there exists %known lower and upper bounds, $\underline\Gamma^{m_{t_k},Y_{t_k},t_k}(\varphi)$ %and $\overline\Gamma^{m_{t_k},Y_{t_k},t_k}(\varphi)$ for $\E^{m_{t_k},Y_{t_k},t_k}\left %[\varphi(m_{t_{k+1}},Y_{t_{k+1}},t_{k+1})\right].$
Since we, by assumption and through the space localization in Step 3, consider continuous pay-off functions and switching costs on bounded domains, as well as a finite horizon problem, there exist lower and upper bounds for \eqref{reg1}. However, a sound way of finding an approximation of the conditional expectation in \eqref{reg1} is needed, and this approximation is the focus of this step of the numerical scheme proposed. To perform an approximation of the conditional expectation in \eqref{reg1}, we make use of an empirical least square regression model based on simulation. In particular, we consider a test function $\varphi$
and the function
\begin{equation}\label{eq:regressionsimplenotation}
F^{m_{t_k},Y_{t_k},t_k}(\varphi):=\E^{m_{t_k},Y_{t_k},t_k}\biggl [\varphi(m_{t_{k+1}},Y_{t_{k+1}},t_{k+1})\biggr ].
\end{equation}
 We will use a least square regression onto a set of $R$ preselected basis functions, $\{B_r\}_{1\leq r \leq R}$, to create an estimator $\hat F^{m_{t_k},Y_{t_k},t_k}(\varphi)$ of $F^{m_{t_k},Y_{t_k},t_k}(\varphi)$. To elaborate on this, assume that we are given $R$ basis functions $\{B_r(m,y,t)\}_{1\leq r \leq R}$. Given the test function $\varphi$, we define $\hat \lambda_{t_k}(\varphi) =(\hat\lambda_{t_k,1}(\varphi) ,....,\hat\lambda_{t_k,R}(\varphi) ) \in \mathbb \R^R$ as
\begin{equation}\label{rego-}
\hat \lambda_{t_k}(\varphi) := \argmin _{ (\lambda_1,...,\lambda_R)}\E^{m_{t_k},Y_{t_k},t_k} \biggl [ \l (\sum_{r=1}^R \lambda_r B_r(m_{t_k},Y_{t_k},t_k) - \varphi(m_{t_{k+1}},Y_{t_{k+1}},t_{k+1} ) \r )^2 \biggr ],
\end{equation}
and set
\begin{equation} \label{eq:Fhat}
\hat  F^{m_{t_k},Y_{t_k},t_k}(\varphi) :=  \sum_{r=1}^R\hat  \lambda_{t_k ,r}(\varphi) B_r(m_{t_k},Y_{t_k},t_k).
\end{equation}
Recall that in our numerical scheme, $m=\bar m^\epsilon$ and  $Y=\bar Y^\epsilon$. Based on this we define a new set of functions $\{\hat v_i^{\check \Pi,\ep}\}$ through the recursive scheme
\begin{eqnarray} \label{eq:scheme2}
\hat v_i^{\check \Pi,\e}(m_T,Y_T,T)&=&0,\notag\\
\hat v_i^{\check \Pi,\e}(m_{t_k},Y_{t_k},t_k) &=& \max _{j \in \Q} \biggl \{\delta \pi^{m_{t_{k}}}_{t_k}(f_{j}(\cdot, Y_{t_k},t_k))+\hat F^{m_{t_k},Y_{t_k},t_k}(\hat v_j^{\check\Pi,\e})\notag\\
&&\qquad\quad -   \pi^{m_{t_{k}}}_{t_k}(c_{i,j}(\cdot, Y_{t_k},t_k))\biggr\}.
\end{eqnarray}
In particular, we obtain a new approximation $\hat v_i^{\check \Pi,\e}$ of the true value function, 
%and the functions $\{\hat v_i^{\check \Pi,\e}\}$ are 
defined through the recursive scheme in \eqref{eq:scheme2}. The error $|\hat v_i^{\check\Pi,\e}(m,y,t_k)-v_i^{\check\Pi,\e}(m,y,t_k)|$ is analyzed in Proposition \ref{prop:regression}.

\begin{remark}
In the above construction, the set of basis functions used is arbitrary and several options are possible. Furthermore, it is possible to choose a different set of basis functions for each time $t_k$ without adding difficulties beyond additional notation. Also, as pointed out in \cite{ACLP12}, one can use stochastic partitions of the domain $D$ to enable the use of adaptive partitioning methods, possibly increasing the convergence speed of the numerical scheme.
\end{remark}

 \subsection{Step 5 -- Replacing the true regression by a sample mean}\label{sec.sample_mean_approx}
 Given a test function $\varphi$,  in numerical calculations the true regression parameters $$\hat \lambda_{t_k}(\varphi) =(\hat\lambda_{t_k,1}(\varphi) ,....,\hat\lambda_{t_k,R}(\varphi) )$$ will not be known, and they have to be replaced by a sample mean $$\tilde\lambda_{t_k}(\varphi) =(\tilde\lambda_{t_k,1}(\varphi) ,....,\tilde\lambda_{t_k,R}(\varphi) )$$ based on simulations. Recall
 that in the context of \eqref{prob} we in the end want to approximate $v_i(m,y,0)=v_i(m,\theta,y,0)$ for $(m,\theta,y)\in\mathbb R^{N_1}\times \mathbb R^{N_1\times N_1} \times \mathbb R^{N_2}$ given and fixed. In particular, this means that in the original model for the process $(m_t,Y_t)$ we consider the initial condition $(m,y)$, at $t=0$. This also implies that the approximations of $(m_t,Y_t)$ introduced above, $(\bar m_t,\bar Y_t)$ and $(\bar m_t^\e,\bar Y_t^\e)$, will also start at $(m,y)$ at $t=0$. To outline the estimation of the true regression parameters we let $\l (\{Y_{t}^\ell\}_{t \in \Pi}\r)_\ell$, ${1\leq \ell \leq M}$, denote $M$ simulated trajectories of the observed process $Y_t$, starting at $y$ at $t=0$, and we use these to calculate the corresponding values of $\l (\{m_t^\ell\}_{t\in \Pi} \r)_\ell$, $1\leq \ell \leq M$ as outlined in Section \ref{sec:KB} and Subsection \ref{subsec:eulerdisc}, with initial condition
 $m_0^\ell=m$ for all $\ell$. Based on the paths $\l (\{Y_{t}^\ell\}_{t \in \Pi}\r)_\ell$ and $\l (\{m_t^\ell\}_{t\in \Pi} \r)_\ell$, $1\leq \ell \leq M$, we now compute the empirical vector $\tilde\lambda_{t_k}(\varphi) =(\tilde\lambda_{t_k,1}(\varphi) ,....,\tilde\lambda_{t_k,R}(\varphi) )$, estimating $\hat \lambda_{t_k}(\varphi) =(\hat\lambda_{t_k,1}(\varphi) ,....,\hat\lambda_{t_k,R}(\varphi) ),$ as
\begin{equation} \label{rego}
\tilde \lambda_{t_k}(\varphi) := \argmin _{ (\lambda_1,...,\lambda_R) } \frac{1}{M}  \sum_{\ell=1}^M \left (  \sum_{r=1}^R \l (\lambda_{r} B_r(m^\ell_{t_k},Y_{t_k}^\ell,t_k) - \varphi(m^\ell_{t_{k+1}},Y_{t_{k+1}}^\ell,t_{k+1} )  \r )^2\right).
\end{equation}
Given $\varphi$, we in this way fix $\tilde\lambda_{t_k}(\varphi)$ for $k\in\{0,...,N-1\}$. In the following we indicate that the estimation of
 $\tilde\lambda_{t_k}(\varphi)$ is based on $M$ sample paths by writing  $\tilde\lambda_{t_k}^M(\varphi)$ instead of $\tilde\lambda_{t_k}(\varphi)$. Next, using $\tilde \lambda_{t_k}^M(\varphi)$ we set
\begin{equation}\label{eq:Ftilde}
\tilde F_M^{m,y,t_k}(\varphi) :=  \sum_{r=1}^R\tilde  \lambda^M_{t_k ,r}(\varphi) B_r(m,y,t_k)
\end{equation}
whenever $(m,y)\in\mathbb R^{N_1}\times \mathbb R^{N_2}$. In particular, while the coefficients $\{\lambda^M_{t_k ,r}(\varphi)\}$ are estimated based on a finite set of sample paths  $\l (\{Y_{t}^\ell\}_{t \in \Pi}\r)_\ell$, $\l (\{m_t^\ell\}_{t\in \Pi} \r)_\ell$, $1\leq \ell \leq M$, we use these coefficients in
\eqref{eq:Ftilde} to construct an estimator for the conditional expectation
$$E^{m,y, t_k}\biggl [\varphi(m_{t_{k+1}},Y_{t_{k+1}},t_{k+1})\biggr ]$$
for all $(m,y)\in\mathbb R^{N_1}\times \mathbb R^{N_2}$. Based on $F_M$ we let  $\tilde v_{i}^{\check \Pi,\e,M}(m_{t_k},Y_{t_k},t_k)$ be defined through the recursive scheme
\begin{align} \label{eq:scheme3}
\tilde v_{i}^{\check \Pi,\e,M}(m_T,Y_T,T)&=0 \notag \\
\tilde v_{i}^{\check \Pi,\e,M}(m_{t_k},Y_{t_k},t_k) &= \max _{j \in \Q} \bigg \{\delta \pi^{m_{t_{k}}}_{t_k}(f_{j}(\cdot, Y_{t_k},t_k))+ \tilde F^{m_{t_k},Y_{t_k},t_k}_M(\tilde v_{j}^{\check\Pi,\e,M}) \notag\\
&\qquad\qquad -  \pi^{m_{t_{k}}}_{t_k}(c_{i,j}(\cdot, Y_{t_k},t_k))\bigg\}.
\end{align}
Then $\tilde v_i^{\check\Pi,\e,M}(m,y,t_k)$ is an approximation of $\hat v_i^{\check\Pi,\e}(m,y,t_k)$ and, since $(m_{0},Y_{0},0)=(m,y,0)$, $\tilde v_i^{\check\Pi,\e,M}(m,y,0)$ is an approximation of
$\hat v_i^{\check\Pi,\e}(m,y,0)$. Hence, the final value produced by the algorithm, $ \tilde v_i^{\check\Pi,\e,M}(m,y,0)$, is an approximation of the true value $v_i(m,y,0)$.

Following  \cite{ACLP12}, we will in this paper use indicator functions on hypercubes as basis functions, $\{B_r=B_r(m,y,t)\}_{1\leq r \leq R}$, for our regression and subsequent error analysis. These hypercubes are defined in relation to  the space (time) localization domain  $D^\e=\{D^\e_{t}\}_{t\in[0,T]}$. Here we briefly outline the idea of using such a basis for regression but we also refer to \cite{ACLP12}. Recall that $D^\e$ is assumed to have the structure specified in \eqref{dom}. We let $\{ B_r\}_{1 \leq r \leq R}$ be a partition of the bounded domain $D^\e$ into $R$ hypercubes, i.e., we split $D^\e$ in to $R$ open hypercubes $B_r$ such that $\cup_{r=1}^R B_r  = D^\e$ and $B_l \cap B_j = \emptyset$ if $l \neq j$. Furthermore, to achieve notational simplicity, we in the following also assume that each hypercube has side length $\Delta$ in each dimension.  Using $\{ B_r\}_{1 \leq r \leq R}$ we define
basis functions to be used in the regression as
$$B_r(m_{t_k},Y_{t_k},t_k) = \I_{\{(m_{t_k},Y_{t_k}) \in B_r\}}, $$
for $1 \leq r\leq R$, and $(m_{t_k}, Y_{t_k}) \in D^\ep$. By definition $\I_{\{(m,Y) \in B_r\}}=1$ if $(m,Y) \in B_r$ and $\I_{\{(m,Y) \in B_r\}}=0$ otherwise. Note that we use the same symbol $B_r$ to denote both the $r$-th hypercube and its corresponding basis function/indicator function. Since conditional expectation is mean-square error-minimizing, this choice of basis functions reduces the vectors \eqref{rego-} and \eqref{rego} to
\begin{align}\label{eq:truemean}
\hat \lambda_{{t_k},r}(\hat v_i^{\check\Pi,\e}) :&= \E \l [\hat v_i^{\check\Pi,\e}(m_{t_{k+1}},Y_{t_{k+1}},t_{k+1}) \, \vline \, (m_{t_k},Y_{t_k})\in B_r\r]\notag\\
&=   \frac{\E \l[\hat v_i^{\check\Pi,\e} (m_{t_{k+1}},Y_{t_{k+1}},t_{k+1}) \I_{\{(m_{t_k},Y_{t_k}) \in B_r \}}\r]}{\P((m_{t_k},Y_{t_k}) \in B_r)},
\end{align}
and
\begin{align}%\label{eq:samplemean}
\tilde \lambda^M_{{t_k},r}(\tilde v_i^{\check\Pi,\e,M}) :&= \frac{\frac{1}{M} \sum_{ \ell=1}^M\l[\tilde v_i^{\check\Pi,\e,M} (m^\ell_{t_{k+1}},Y^\ell_{t_{k+1}},t_{k+1}) \I_{\{(m^\ell_{t_k},Y^\ell_{t_k}) \in B_r \}}\r]}{\frac{1}{M} \sum_{\ell=1}^M \I_{\{(m^\ell_{t_k},Y^\ell_{t_k}) \in B_r \}}} \notag
\end{align}
with the convention that $\tilde \lambda^M_{{t_k},r}(\tilde v_i^{\check\Pi,\e,M})=0$ if, at time $t_k$, no path $(m^\ell,Y^\ell)$ lies inside the hypercube $B_r$. Hence, in our numerical scheme, if  $(m_{t_k},Y_{t_k}) \in B_r$ the expected future value $$\hat \lambda_{{t_k},r}(\hat v_i^{\check\Pi,\e})=\E \l [\hat v_i^{\check\Pi,\e} (m_{t_{k+1}},Y_{t_{k+1}},t_{k+1}) \, \vline \, (m_{t_k},Y_{t_k})\in B_r\r]\ (i \in \Q)$$
 will be approximated by $\tilde \lambda^M_{{t_k},r}(\tilde v_i^{\check\Pi,\e,M})$. The error $|\hat v_i^{\check\Pi,\e}(m,y,t_k)-\tilde v_i^{\check\Pi,\e,M}(m,y,t_k)|$
is estimated in Proposition \ref{lemma:truetosample}.

\setcounter{equation}{0} \setcounter{theorem}{0}
\section{Convergence analysis} \label{sec:convanalysis}
\noindent
In this section we establish the convergence of the numerical scheme outlined in the previous section by proving Theorem \ref{thm:convergence} stated below. However, we first  recall
degrees of freedom, at our disposal, in the numerical scheme proposed.
\begin{itemize}
\item $\delta=T/N$ -- the time discretization parameter,
\item $\e$ -- the error tolerance when choosing the bounded convex domains
$D_t^\e$ for the projection,
\item $\Delta$ -- the edge size of  the hypercubes used in the regression,
\item $M$ -- the number of simulated trajectories of $\{Y_{t}^\ell\}_{t \in \Pi^\delta}$ and
$\{m_t^l\}_{t\in \Pi^\delta}$,  used in \eqref{rego}.
\end{itemize}
Recall that $(\Omega,\mathcal F,\mathbb P)$ is the underlying probability space and in the following $L^2=L^2(\Omega,d\mathbb P)$ with norm $||\cdot||_{L^2}$. We prove the following  convergence theorem.
\begin{theorem} \label{thm:convergence} Assume that $(m_0, \theta_0, Y_0)=(m,\theta,y)\in \mathbb R^{N_1}\times \mathbb R^{N_1\times N_1} \times \mathbb R^{N_2}$ and that all assumptions and conditions stated and used in the previous sections are fulfilled. Then there exist a constant $C_1$, independent of
$\delta$, $\epsilon$, $\Delta$ and $M$, and a constant $C_2$, independent of
$\delta$,  $\Delta$ and $M$,  such that
\begin{align}
&\l \| \max_{i \in \Q} \Big|v_i(m,\theta,y,0)-\tilde v_i^{\check\Pi,\e,M}(m,\theta,y,0)\Big| \r \|_{L^2} \notag \\
\leq & C_1 \biggl \{ \l( \delta \log(\frac {2T}{\delta})\r )^{1/2} +\delta^{1/2}   + \delta  +
 \e + \frac{\Delta}{\delta}\notag\\
 & + \frac{1+C_2}{\delta \sqrt{Mp_{min}(T,\Delta,\e)}}
 +\frac{1+C_2}{\delta{Mp_{min}(T,\Delta,\e)}}\biggr\},\notag
\end{align}
where
\begin{equation*}
p_{min}(T,\Delta,\e):=\min_{t\in\Pi^\delta\cap[0,T]}\min_{B_r\subset D^\e}\mathbb P[(m_{t},Y_t)\in B_r ]
\end{equation*}
is a strictly positive quantity. In particular, if $\e \to0$, $\delta \to 0$, $\Delta \to 0$ and $M \to \infty$
such that $$ \frac{\Delta}{\delta} \to 0,\ \frac{1+C_2}{\delta \sqrt{Mp_{min}(T,\Delta,\e)}}
\to 0,\mbox{ and }\frac{1+C_2}{\delta{Mp_{min}(T,\Delta,\e)}}
\to 0,$$ then $$v_i^{\check\Pi,\e,M}(m,\theta,y,0) \xrightarrow{L^2}  v_i(m,\theta,y,0)\mbox{ uniformly in $i\in\Q$}.$$
\end{theorem}

We first note, using the notation introduced
in the previous section, that
\begin{eqnarray*}%\label{err1}
v_i(m,y,t_k)-\tilde v_i^{\check\Pi,\epsilon,M}(m,y,t_k)&=&E_1(m,y,t_k)+E_2(m,y,t_k)+E_3(m,y,t_k)\notag\\
&&+E_4(m,y,t_k)+E_5(m,y,t_k)
\end{eqnarray*}
where
\begin{eqnarray*}%\label{err2}
E_1(m,y,t_k)&:=&v_i(m,y,t_k)-v_i^\Pi(m,y,t_k),\notag\\
E_2(m,y,t_k)&:=&v_i^\Pi(m,y,t_k)-v_i^{\check\Pi}(m,y,t_k),\notag\\
E_3(m,y,t_k)&:=&v_i^{\check\Pi}(m,y,t_k) - v_i^{\check\Pi,\e}(m,y,t_k), \notag \\
E_4(m,y,t_k)&:=&v_i^{\check\Pi,\e}(m,y,t_k)-\hat v_i^{\check\Pi,\e}(m,y,t_k),\notag\\
E_5(m,y,t_k)&:=&\hat v_i^{\check\Pi,\e}(m,y,t_k)-\tilde v_i^{\check\Pi,\e}(m,y,t_k).
\end{eqnarray*}
In the subsequent subsections we prove that the errors $E_1-E_5$ can be controlled and we ending the section with a summary proving Theorem \ref{thm:convergence}.
\subsection{Preliminary lemmas}
Before deducing the relevant error estimates, i.e., quantifying  $E_1$ to $E_5$, we here state and prove some auxiliary results, Lemma \ref{prop:finitenumberofswitches} -- Lemma  \ref{lemma:Lipschitzy}, which will be used in the subsequent proofs.

\begin{lemma}\label{prop:finitenumberofswitches} Assume \eqref{assump1elel}, \eqref{ass:functions2}, \eqref{ass:switchingcosts}, and let
$\mu^\ast=(\{\tau_k^\ast\}_{k\geq 0},\{\xi_k^\ast\}_{k\geq 0})$ be such that
\begin{eqnarray*}
\tilde J(\mu^\ast)=\sup_{\mu}\tilde J(\mu),\mbox{ where $\tilde J(\mu)$ is defined in \eqref{eq1lu}}.
\end{eqnarray*}
Let  $N(\mu^\ast)= |\{ n : \tau_n^\ast \leq T\}|$. Then
\begin{equation*}
N(\mu^\ast) \leq \frac {2 T\sup_i \|f_i\|_\infty}{\nu},
\end{equation*}
where $\nu>0$ is the positive constant appearing in \eqref{ass:switchingcosts} $(ii)$. In particular, the number of switches in an optimal strategy $\mu^\ast$ is finite.
\end{lemma}
\begin{proof} Let $\mu_0\in \A_{t,i}^Y$ denote the trivial strategy, i.e., no switches. Then, using \eqref{assump1elel}, \eqref{ass:functions2} and \eqref{ass:switchingcosts}, we see that
\begin{eqnarray*}
-T\sup_i \|f_i\|_\infty\leq \tilde J(\mu_0)\leq T\sup_i \|f_i\|_\infty.
\end{eqnarray*}
Now, let $\mu_{\infty} \in \A_{t,i}^Y$ be a strategy with an unbounded number of switches. Then, by \eqref{ass:functions2} $(i)$ and \eqref{ass:switchingcosts} $(ii)$, it follows that $\tilde J(\mu_\infty) = - \infty$. Hence, the optimal strategy $\mu^\ast$ must consist of a finite number of switches. Finally, let $\mu\in \A_{t,i}^Y$ be an arbitrary strategy with a finite number of switches and assume that $\tilde J(\mu)\geq \tilde J(\mu_0)$. Then, using \eqref{ass:functions2} $(i)$ and \eqref{ass:switchingcosts} $(ii)$ we see that
$$0\leq \tilde J(\mu) - \tilde J(\mu_0) \leq 2 T\sup_i \|f_i\|_\infty-\nu N(\mu).$$
In particular, $N(\mu^\ast) \leq  {2 T \sup_i \|f_i\|_\infty}/{\nu}$.
\end{proof}

\begin{lemma}\label{prop:BMLPnorm}
Let $\{W_t\}$ be a $d$-dimensional Brownian motion defined on a filtered probability space $(\Omega, \F, \{\F_t\}_{t\geq 0}, \P)$. Then, for any $p \geq 1$ there exists a finite constant $C_p>0$ such that
\begin{equation*}
\E\l[ \l (  \sup \limits_{\substack{0\leq t,s \leq T \\ |t-s|\leq \delta }} |W_t - W_s|  \r)^p \r ]  \leq C_p \l ( \delta \log(\frac {2T}{\delta})\r )^{p/2}.
\end{equation*}
\end{lemma}
\begin{proof} This is Theorem 1 in \cite{FN10}.\end{proof}

\begin{lemma}\label{lemma:Lipschitzm} Assume \eqref{assump1elel}, \eqref{ass:functions2}, \eqref{ass:switchingcosts} and \eqref{ass:switchingcosts+}. Then, there exists a constant $C_m$, independent of $\delta$,  such that
\begin{eqnarray*}%\label{ind1-}
|v_i^{\check \Pi,\e}(m_1,Y_{t_k},t_k) - v_i^{\check \Pi,\e}(m_2,Y_{t_k},t_k) | \leq  C_{m} |m_1 - m_2|.
\end{eqnarray*}
whenever $(i,m_1,Y_{t_k},t_k), (i,m_2,Y_{t_k},t_k) \in \Q \times \R^{N_1} \times \R^{N_2} \times \Pi$.\end{lemma}
\begin{proof} We claim that there exist positive constants $A$ and $L$, independent of $\delta$,  such that the following holds. There exists a sequence of constants $\{E_{m,k}\}_{k=0}^N$ such that
\begin{eqnarray}\label{ind1}
|v_i^{\check \Pi,\e}(m_1,Y_{t_k},t_k) - v_i^{\check \Pi,\e}(m_2,Y_{t_k},t_k) | \leq  E_{m,k} |m_1 - m_2|,
\end{eqnarray}
whenever $(i,m_1,Y_{t_k},t_k), (i,m_2,Y_{t_k},t_k) \in \Q \times \R^{N_1} \times \R^{N_2} \times \Pi$ and such that 
\begin{align}\label{Iprin+aa}
E_{m,N}&=0\notag\\
E_{m,k}&=A\delta+E_{m,k+1}(1+L\delta).
\end{align}
To prove this we proceed by (backward) induction on $k$ and we let $I(k)=1$ if \eqref{ind1} holds with a constant  $E_{m,k}$ whenever $(i,m_1,Y_{t_k},t_k), (i,m_2,Y_{t_k},t_k) \in \Q \times \R^{N_1} \times \R^{N_2} \times \Pi$, and if $E_{m,k}$ is related to the bounds $E_{m,k+1}, \dots, E_{m,N}$ as stated in
\eqref{Iprin+aa}. We want to prove that $I(k)=1$ whenever $k\in\{0,\dots,N\}$. Since
 $v_i^{ \check \Pi,\e}(m, Y_T,T) =0$ for all $m$, we immediately see that $I(N)=1$. Assuming that
 $I(k+1)=1$ for some $k\in[0,N-1]$ we next prove that $I(k)=1$ by constructing $E_{m,k}$. To do this we first note, simply by the (discrete) DPP, that
\begin{align}
 v_i^{\check \Pi,\e} (m_1,Y_{t_k},t_k)&=\max_{j \in \Q} \bigg\{ \delta \pi_{t_k}^{m_1} \l (f_j(\cdot,Y_{t_k},t_k) \r) + \E^{m_1,Y_{t_k},t_k} \l[v_j^{\check \Pi,\e}(m_{t_{k+1}},Y_{t_{k+1}},t_{k+1})  \r]\notag\\
  &\qquad\qquad- \pi_{t_k}^{m_1} \l ( c_{i,j}(\cdot,Y_{t_k},t_k) \r)  \notag \\
&+  (1-1)\biggl[ [ \delta \pi_{t_k}^{m_2} \l (f_j(\cdot,Y_{t_k},t_k) \r) + \E^{m_2,Y_{t_k},t_k} \l[ v_j^{\check \Pi,\e}(m_{t_{k+1}},Y_{t_{k+1}},t_{k+1})  \r]\notag\\
&\qquad\qquad - \pi_{t_k}^{m_2} \l ( c_{i,j}(\cdot,Y_{t_k},t_k) \r) \biggr]  \bigg \}. \notag
\end{align}
Furthermore, by elementary manipulations the above can be rewritten as
\begin{align}
 v_i^{\check \Pi,\e} (m_1,Y_{t_k},t_k) &=\max_{j \in \Q} \bigg \{ \delta \pi_{t_k}^{m_2} \l (f_j(\cdot,Y_{t_k},t_k) \r)  +  \E^{m_2,Y_{t_k},t_k} \l[v_j^{\check \Pi,\e}(m_{t_{k+1}},Y_{t_{k+1}},t_{k+1})  \r]\notag\\
 &\qquad\qquad - \pi_{t_k}^{m_2} \l ( c_{i,j}(\cdot, Y_{t_k},t_k) \r ) \notag \\
&\qquad\qquad+ \delta \l (\pi_{t_k}^{m_1} \l (f_j(\cdot, Y_{t_k},t_k) \r) - \pi_{t_k}^{m_2} \l (f_j(\cdot, Y_{t_k},t_k) \r) \r )\notag \\
&\qquad\qquad+ \E^{ m_1,Y_{t_k},t_k} \l[v_j^{\check \Pi,\e}(m_{t_{k+1}},Y_{t_{k+1}},t_{k+1})  \r]\notag\\
 &\qquad\qquad-\E^{m_2,Y_{t_k},t_k} \l[ v_j^{\check \Pi,\e}(m_{t_{k+1}},Y_{t_{k+1}},t_{k+1})  \r]  \notag \\
&\qquad\qquad- \l ( \pi_{t_k}^{m_1} \l ( c_{i,j}(\cdot,Y_{t_k},t_k) \r)  - \pi_{t_k}^{m_2} \l ( c_{i,j}(\cdot,Y_{t_k},t_k ) \r ) \r)  \bigg \}. \notag
\end{align}
In particular,
\begin{align}
 |v_i^{\check \Pi,\e} (m_1,Y_{t_k},t_k) -{v_i^{\check \Pi,\e} ( m_2, Y_{t_k},t_k)}|\leq A_1 + A_2 + A_3\notag,
\end{align}
where
\begin{align}
A_1 =& \max_{j \in \Q} \bigg \{\biggl |\delta \l (\pi_{t_k}^{m_1} \l (f_j(\cdot, Y_{t_k},t_k) \r) - \pi_{t_k}^{m_2} \l (f_j(\cdot, Y_{t_k},t_k) \r) \r )\biggr|\biggr\}, \notag \\
A_2 =& \max_{j \in \Q} \bigg \{\biggl |\E^{ m_1,Y_{t_k},t_k} \l[v_j^{\check \Pi,\e}(m_{t_{k+1}},Y_{t_{k+1}},t_{k+1})  \r]\notag\\
 &\qquad\qquad-\E^{m_2,Y_{t_k},t_k} \l[ v_j^{\check \Pi,\e}(m_{t_{k+1}},Y_{t_{k+1}},t_{k+1})  \r] \biggr|\biggr\}, \notag \\
A_3 =&  \max_{j \in \Q} \bigg \{\biggl |\pi_{t_k}^{m_2} \l ( c_{i,j}(\cdot,Y_{t_k},t_k ) \r ) - \pi_{t_k}^{m_1} \l ( c_{i,j}(\cdot,Y_{t_k},t_k) \r)\biggr|\biggr\}. \notag
\end{align}
We now need to bound the terms $A_1, A_2$ and $A_3$ with $c |m_1 -m_2|$. We first treat  the term $A_1$ using the Lipschitz property of $f_j$. Indeed, recalling  Corollary \ref{cordir} we first have
\begin{align}
&\pi_t^{m_1}(f_j) - \pi_t^{m_2}(f_j) \notag\\
 &=\frac{1}{(2\pi)^{N_1/2}}  \l( \int _{\R^{N_1}} \l(f_j (m_1 + \theta_t^{1/2} z) -f_j (m_2 + \theta_t^{1/2} z) \r) \exp(- \dfrac{|z |^2}{2}) d z \r).\notag
\end{align}
Hence, by the Lipschitz property of $f_j$ we see that
\begin{align}\label{eq:pilipschitz}
&|\pi_{t_k}^{m_1}(f_j) - \pi_{t_k}^{m_2}(f_j)| \leq \frac{c_1}{(2\pi)^{N_1/2}}\int _{\R^{N_1}} |m_1 - m_2| \exp( -\dfrac{|z |^2}{2}) dz \leq c_1 |m_1- m_2|
\end{align}
where $c_1$ is simply the Lipschitz constant of $f_j$. In particular,
\begin{eqnarray*} % \label{a1}
A_1\leq \delta c_1|m_1-m_2|.
\end{eqnarray*}
To treat the term $A_2$ we first note that
\begin{eqnarray*}%\label{est1}
&&\biggl |\E^{ m_1,Y_{t_k},t_k} \l[v_j^{\check \Pi,\e}(m_{t_{k+1}},Y_{t_{k+1}},t_{k+1})  \r]-\E^{m_2,Y_{t_k},t_k} \l[ v_j^{\check \Pi,\e}(m_{t_{k+1}},Y_{t_{k+1}},t_{k+1})  \r] \biggr|^2\notag\\
&&\leq \biggl\|v_j^{\check \Pi,\e}(m_{t_{k+1}}^{ m_1,t_k},Y_{t_{k+1}}^{Y_{t_k},t_k},t_{k+1})-v_j^{\check \Pi,\e}(m_{t_{k+1}}^{m_2,t_k},Y_{t_{k+1}}^{Y_{t_k},t_k},t_{k+1}) \biggr\|_{L^2}^2
\end{eqnarray*}
where $(m_{t_{k+1}}^{ m_1,t_k},Y_{t_{k+1}}^{Y_{t_k},t_k})$ and $(m_{t_{k+1}}^{ m_2,t_k},Y_{t_{k+1}}^{Y_{t_k},t_k})$ indicate that at $t_k$ the processes $m$ and $Y$ are starting at $(m_1,Y_{t_k})$ and $(m_2,Y_{t_k})$, respectively. Now, using the induction hypothesis $I(k+1)=1$, i.e., the Lipschitz property of  $v_i^{\check \Pi,\e}(\cdot,Y_{t_{k+1}},t_{k+1})$ and $v_j^{\check \Pi,\e}(\cdot,Y_{t_{k+1}},t_{k+1})$, we can conclude that
\begin{eqnarray}\label{est1+}
A_2\leq E_{m,k+1}\biggl\|m_{t_{k+1}}^{ m_1,t_k}-m_{t_{k+1}}^{m_2,t_k}\biggr\|_{L^2}.
\end{eqnarray}
Next, using \eqref{est1+}, the equation for $m_t$ and elementary estimates, e.g., see the proof of Lemma $3.3$ in \cite{ACLP12}, we can conclude that
\begin{eqnarray*}%\label{est1++}
A_2 \leq E_{m,k+1}(1+L\delta) |m_1-m_2|
\end{eqnarray*}
for some constant $L$ independent of $\delta$, $m_1$ and $m_2$. Finally, using the assumption in \eqref{ass:switchingcosts+} we see that $A_3=0$.  Putting the estimates together we can conclude that
\begin{eqnarray*}%\label{est1+++}
 |v_i^{\check \Pi,\e} (m_1,Y_{t_k},t_k) -  {v_i^{\check \Pi,\e} ( m_2, Y_{t_k},t_k)}| &\leq &A_1+A_2+A_3\notag\\
 &\leq&(\delta c_1+E_{m,k+1}(1+L\delta))|m_1-m_2|\notag\\
 &=:&E_{m,k}|m_1-m_2|
 \end{eqnarray*}
 where we have chosen $A=c_1$ in the definition of $E_{m,k}$. In particular, if $I(k+1)=1$ for some $k\in\{0,..,N-1\}$, then $I(k)=1$ and hence the lemma follows by induction.
\end{proof}

\begin{lemma}\label{lemma:Lipschitzy} Assume \eqref{assump1elel}, \eqref{ass:functions2} \eqref{ass:switchingcosts} and \eqref{ass:switchingcosts+}. Then, there exists a constant $C_Y$, independent of $\delta$,  such that
\begin{eqnarray*}%\label{ind2}
|v_i^{\check \Pi,\e}(m_{t_k},y_1,t_k) - v_i^{\check \Pi,\e}(m_{t_k},y_2,t_k) | \leq  C_{Y} |y_1 - y_2|.
\end{eqnarray*}
whenever $(i,m_{t_k},y_1,t_k), (i,m_{t_k},y_2,t_k) \in \Q \times \R^{N_1} \times \R^{N_2} \times \Pi$.
\end{lemma}
\begin{proof} The proof of the lemma is analogous to the proof of Lemma \ref{lemma:Lipschitzm} and we here omit further details.
\end{proof}

\begin{remark} We emphasize that in this paper the structural assumption \eqref{ass:switchingcosts+} is only used in the proof of Lemma \ref{lemma:Lipschitzm} and Lemma \ref{lemma:Lipschitzy}. These lemmas are then only used in the proof of Proposition \ref{prop:regression} stated below. In particular, 
if Lemma \ref{lemma:Lipschitzm} and Lemma \ref{lemma:Lipschitzy} can be proved without assuming \eqref{ass:switchingcosts+} all of the remaining arguments go through unchanged.
\end{remark}

\subsection{Step 1: Controlling $E_1$} %\label{sec:controlE1}
%In the following we denote by $\A^{Y,K}_{t,i}$ the set of strategies with at most $K$ switches.
\begin{proposition} \label{lemma:bermudan}There exist positive constants $C_{11}$ and $C_{12}$, independent of  $\delta$, such that
$$|E_1|=| v_i(m_{t_k},Y_{t_k},t_k) - v_i^{\Pi}(m_{t_k},Y_{t_k},t_k) | \leq C_{11}\l( \delta \log(\frac {2T}{\delta})\r )^{1/2}+C_{12}\delta$$
whenever $(i,m_{t_k},Y_{t_k},t_k) \in \Q \times \R^{N_1} \times \R^{N_2} \times \Pi$.
\end{proposition}
\begin{proof} Recall that
\begin{align} %\label{eq:est1}
v_i(m_{t_k},Y_{t_k},t_k) = \sup _{\mu \in \A^{Y}_{t_k,i}} \E ^{m_{t_k},Y_{t_k},t_k}&\biggl [\int \limits _{t_k} ^T \pi^{m_s}_s \l ( f_{\mu_s}(\cdot,Y_{s},s) \r )ds\notag\\
& - \sum_{t_k \leq \tau_n \leq T} \pi_{\tau_{n}}^{m_{\tau_{n}}} \l (c_{ \xi_{n-1}, \xi_{n}}(\cdot, Y_{\tau_{n}},\tau_{n}) \r) \biggr ], \notag
\end{align}
and
\begin{align}\label{eq:bermudanagain}
v_i^{\Pi}(m_{t_k},Y_{t_k}, t_k)= \sup _{\mu \in \A^{Y,\Pi}_{t_k,i}} \E^ {m_{t_k},Y_{t_k}, t_k} &\biggl [\int_{t_k}^T \pi_s ^{m_{s}}\l ( f_{\mu_{s}}(\cdot, Y_s,s) \r) ds\notag\\
& - \sum _{t_k \leq \tau_n \leq T} \pi_{\tau_{n}}^{m_{\tau_{n}}} \l (c_{\xi_{n-1},\xi_{{n}}}( \cdot, Y_{\tau_{n}},\tau_{n})\r )\biggr ]. 
\end{align}
Using Lemma \ref{prop:finitenumberofswitches} we may assume that  $N(\mu)<\infty$. Given $\mu=(\{\tau_n\}_{n\geq 0},\{\xi_n\}_{n\geq 0}) \in \A^Y_{t,i}$
we consider the strategy $\tilde \mu = (\{\tilde \tau_n\}_{n\geq 0},\{\tilde \xi_n\}_{n\geq 0})$,  where
\begin{align} %\label{eq:est1apa}
\tilde \tau_n = \min \{t \in \Pi:  t \geq \tau_n\}, \quad \tilde \xi_n = \xi_n. \notag
\end{align}
Then $\tilde \mu\in\A^{Y,\Pi}_{t_k,i}$ and it is the Bermudan strategy associated  to $\mu \in \A^{Y}_{t_k,i}$ $\tilde \mu\in\A^{Y,\Pi}_{t_k,i}$, see Definition \ref{def:approximatingstrategy}. Using this and \eqref{eq:bermudanagain} we see that
\begin{align} %\label{eq:est1+}
 v_i^{\Pi}(m_{t_k},Y_{t_k},t_k)\geq \sup _{\mu \in \A^{Y}_{t_k,i}} \E ^{m_{t_k},Y_{t_k},t_k}&\biggl [\int \limits _{t_k} ^T \pi^{m_s}_s \l ( f_{\tilde\mu_s}(\cdot,Y_{s},s) \r )ds\notag\\
& - \sum_{t_k \leq \tilde \tau_n \leq T} \pi_{\tilde\tau_{n}}^{m_{\tilde\tau_{n}}} \l (c_{ \tilde\xi_{n-1}, \tilde\xi_{n}}(\cdot, Y_{\tilde\tau_{n}},\tilde\tau_{n}) \r) \biggr ]. \notag
\end{align}
Furthermore, since ${\A^{Y,\Pi}_{t_k,i}} \subset {\A^Y_{t_k,i}}$ we also have that $v_i^{\Pi}(m_{t_k},Y_{t_k},t_k) \leq v_{i}(m_{t_k},Y_{t_k},t_k)$. Putting these estimates together we can conclude that
\begin{equation*}
\left | v_{i}(m_{t_k},Y_{t_k},t_k) - v_i^{\Pi}(m_{t_k},Y_{t_k},t_k)   \right | \leq \sup _{\mu \in \A^{Y}_{t_k,i}} \left ( E_{11}^{m_{t_k},Y_{t_k},t_k}(\mu) +  E_{12}^{m_{t_k},Y_{t_k},t_k}(\mu) \right )
\end{equation*}
where
\begin{align}
 E_{11}^{m_{t_k},Y_{t_k},t_k}(\mu) &= \E^{m_{t_k},Y_{t_k},t_k} \biggl [\int _{t_k} ^T \biggl |   \pi^{m_s}_s(f_{\mu_s}(\cdot, Y_s,s))  -
  \pi^{m_s}_s(f_{\tilde \mu_s}(\cdot, Y_s,s) ) \biggr | ds \biggr],  \notag \\
 E_{12}^{m_{t_k},Y_{t_k},t_k}(\mu) &= \E^{m_{t_k},Y_{t_k},t_k} \biggl [\sum_{t_k \leq \tau_n \leq T}\biggl |\pi^{m_{\tau_{n}}}_{\tau_{n}} \l(c_{\xi_{n-1},\xi_{n}}(\cdot, Y_{\tau_{n}},\tau_{n}) \r)\notag\\
 &\qquad\qquad\qquad\qquad\qquad-\pi^{m_{\tilde \t_{n}}}_{\tilde \tau_{n}} \l(c_{\tilde \xi_{n-1},\tilde \xi_{n}}(\cdot, Y_{\tilde \tau_{n}},\tilde \tau_{n}) \r) \biggr | \biggr]\notag.
\end{align}
Since, by assumption, $f_i$ is bounded we immediately see that
\begin{equation*}
E_{11}^{m_{t_k},Y_{t_k},t_k}(\mu)  \leq c \sup_i \|f_i\|_\infty \left (\E ^{m_{t_k},Y_{t_k},t_k} \biggl [ \int _{t_k}^T \chi_{\{\mu_s \neq \tilde \mu_s\}} ds \biggr ] \right).
\end{equation*}
Furthermore, using Lemma \ref{prop:finitenumberofswitches} we see that
$$
\E ^{m_{t_k},Y_{t_k},t_k} \biggl [ \int _{t_k}^T \chi_{\{\mu_s \neq \tilde \mu_s\}} ds \biggr] \leq \delta \E ^{m_{t_k},Y_{t_k},t_k} [ N(\mu)  ] \leq
\biggl [ \frac {2T\sup_i \|f_i\|_\infty}{\nu}\biggr ] \delta.
$$
Putting these estimates together we can conclude that
$$
 E_{11}^{m_{t_k},Y_{t_k},t_k}(\mu) \leq c \biggl [ \frac {2 T (\sup_i \|f_i\|_\infty)^2}{\nu} \biggr ]  \delta
$$
and this gives the appropriate bound on $ E_{11}^{m_{t_k},Y_{t_k},t_k}(\mu)$. To estimate $E_{12}^{m_{t_k},Y_{t_k},t_k}(\mu)$ we first note that
\begin{eqnarray*}
  E_{12}^{m_{t_k},Y_{t_k},t_k}(\mu)\leq  E_{121}^{m_{t_k},Y_{t_k},t_k}(\mu) + E_{122}^{m_{t_k},Y_{t_k},t_k}(\mu)
\end{eqnarray*}
where
\begin{align}
E_{121}^{m_{t_k},Y_{t_k},t_k}(\mu)&=\E ^{m_{t_k},Y_{t_k},t_k}  \biggl [\sum_{t_k \leq \tau_n \leq T}  \biggl |\pi^{m_{\tau_{n}}}_{\tau_{n}}\l (c_{\xi_{n-1},\xi_{n}}(\cdot, Y_{\tau_{n}},\tau_{n})\r)\notag\\
 &\qquad\qquad\qquad\qquad\qquad-\pi^{m_{\tau_{n}}}_{\tau_{n}}\l (c_{\tilde \xi_{n-1},\tilde \xi_{n}}(\cdot, Y_{\tilde \tau_{n}},\tilde \tau_{n})\r ) \biggr | \biggr ],\notag \\
E_{122}^{m_{t_k},Y_{t_k},t_k}(\mu)&=\E ^{m_{t_k},Y_{t_k},t_k} \biggl [\sum_{t_k \leq \tau_n \leq T}\biggl |\pi^{m_{\tau_{n}}}_{\tau_{n}}\l (c_{\xi_{n-1},\xi_{n}}(\cdot, Y_{\tilde \tau_{n}},\tilde \tau_{n})\r )\notag\\
 &\qquad\qquad\qquad\qquad\qquad-\pi^{m_{\tilde \tau_{n}}}_{\tilde \tau_{n}}\l (c_{\tilde \xi_{n-1},\tilde \xi_{n}}(\cdot, Y_{\tilde \tau_{n}},\tilde \tau_{n})\r ) \biggr| \biggr ] \notag.
\end{align}
By using that $(\tilde \xi_{n-1},\tilde \xi_{n})=(\xi_{n-1},\xi_{n})$ by construction, \eqref{ass:functions2} $(ii)$, and that $(Y_{\tilde \tau_{n}},\tilde \tau_{n})$ is $F^Y_{\tau_{n}}$-adapted, we see that there exists a constant $c$ such that
\begin{align} \label{eq:E21est}
 E_{121}^{m_{t_k},Y_{t_k},t_k}(\mu)&\leq c \E ^{m_{t_k},Y_{t_k},t_k} \l[ \sum_{t_k \leq \tau_n \leq T}\left ( \left | \tau_{n} - \tilde \tau_{n} \right | + \left | Y_{ \tau_{n}} - Y_{\tilde \tau_{n}} \right |   \right ) \r]\notag \\
&\leq c \E ^{m_{t_k},Y_{t_k},t_k} \l[  N(\mu)\left ( \delta +   \sup \limits_{\substack{0\leq s,u \leq T \\ |s-u|\leq \delta }} |Y_u - Y_s|    \right ) \r ].
\end{align}
%Since $Y$ is not a standard Browninan motion under $\mathbb{P}$, we have to make the following argument to be able to use Lemma \ref{prop:BMLPnorm}.
We now recall that
$$Y_t=Y_0+\int_0^th(X_r)dr+U_t,$$
where $U_t$ is a standard $m_2$-dimensional Brownian motion. Hence,
\begin{align}
|Y_u - Y_s|\leq|U_u - U_s|+\int_s^uh(X_r)dr\leq|U_u - U_s|+c|u - s| \notag
\end{align}
since $h$ is a bounded function.
Applying Lemma \ref{prop:BMLPnorm} gives
\begin{align}\label{eq:diffY}
\mathbb{E}^{m_{t_k},Y_{t_k},t_k}\l[ \l (  \sup \limits_{\substack{0\leq s,u \leq T \\ |s-u|\leq \delta }} |Y_u - Y_s|  \r)^2 \r ]\leq c\left(\delta\log\left(\frac{2T}{\delta}\right)\right)^{}+c\delta^2.
\end{align}
Combining \eqref{eq:E21est}, \eqref{eq:diffY}, Lemma \ref{prop:finitenumberofswitches},  and the Cauchy-Schwartz inequality we can therefore conclude that
$$
 E_{121}^{m_{t_k},Y_{t_k},t_k}(\mu)\leq  c \l(   \frac {2 T (\sup_i \|f_i\|_\infty)^2}{\nu}  \r) \biggl [\left(\delta\log\l(\frac {2T}{\delta}\r)\right)^{1/2}+\delta\biggr ]
$$
and this completes the estimate of $E_{121}^{m_{t_k},Y_{t_k},t_k}(\mu)$. Hence, it now only remains to prove that $E_{122}^{m_{t_k},Y_{t_k},t_k}(\mu)\leq  c \delta $. To start the estimate, recall that $\pi_t^{m_{t}}$ satisfies the Kushner-Stratonovich equation \eqref{eq.kushner_stratonovich_equation}, i.e., for $0\leq s < t \leq T$ and $\phi\in C^{2}_b(\R^{N_1})$, we have
\begin{equation}\label{eq.pi_probability-}
d\pi^{m_t}_t(\phi)=\pi^{m_t}_t(\H\phi)dt+\sum_{k=1}^{N_2}[\pi^{m_t}_t(h_k\phi)-\pi^{m_t}_t(h_k)\pi(\phi)][dY_t^k-\pi_t^{m_t}(h_k)dt].
\end{equation}
Furthermore, since $\pi^{m_t}_t$ is a probability-measure valued process we have for any such $\phi$ that
\begin{equation*}%\label{eq.pi_probability}
\pi_t^{m_t}(\phi)=\int_{\mathbb{R}^{N_1}}\phi(x)\pi_t^{m_t}(dx)\leq C_\phi\pi_t^{m_t}(\mathbb R^{N_1})=C_\phi
\end{equation*}
where $C_\phi$ denotes the upper bound of $\phi$. We also note, using \eqref{eq.pi_probability-}, that
\begin{align}%\label{eq.control_of_pi_for_different_time}
&\E\biggl [|\pi_t^{m_t}(\phi)-\pi_s^{m_s}(\phi)|\biggr ]\nonumber\\
=&\E\biggl [\int_s^t\left(\pi_u^{m_u}(\mathcal H\phi)-\sum_{k=1}^{N_2}(\pi^{m_u}_u(h_k\phi)-\pi^{m_u}_u(h_k)\pi(\phi))\right)du\biggr ]\nonumber\\
+&\E\biggl [\int_s^t\sum_{k=1}^{N_2}\left(\pi^{m_u}_u(h_k\phi)-\pi^{m_u}_u(h_k)\pi(\phi)\right)dY_u\biggr ]\nonumber\\
\leq&\hat C_\phi|t-s|, \notag
%+C_{\phi,h,m_2}|t-s|^{1/2}\leq C|t-s|^{1/2}.
\end{align}
for yet another constant $\hat C_{\phi}$, independent of $t$ and $s$. Note that essentially $\hat C_{\phi}$ only depends on the
$C^{2}_b(\R^{N_1})$-bounds of $\phi$ and the functions/parameters defining the system for $Y$. Now, applying this, with $t=\tilde \tau_{n}$, $s=\tau_{n}$, recalling Lemma \ref{prop:finitenumberofswitches} and that $c_{i,j}\in C^{2}_b(\R^{N_1})$,  we can conclude that
\begin{align*}%\label{finalest+}
E_{122}^{m_{t_k},Y_{t_k},t_k}(\mu)\leq c\delta.
\end{align*}
This completes the proof of the proposition.
\end{proof}

\begin{remark}
If we have no $X$-dependence on the switching costs, then $E_{122} =0$.
\end{remark}

\subsection{Step 2: Controlling $E_2$}
\begin{proposition}\label{lemma:discretization} There exist positive constants $C_{21}$, $C_{22}$, and $C_{23}$, independent of  $\delta$, such that
\begin{align}
|E_2|=|v^\Pi_i(m_{t_k},Y_{t_k},t_k) - v_i^{\check\Pi}(m_{t_k}, Y_{t_k},t_k) | \leq C_{21} \l ( \delta \log(\frac {2T}{\delta})\r )^{1/2}+C_{22} \delta + C_{23} \delta^{1/2}\notag
\end{align}
whenever $(i,m_{t_k},Y_{t_k},t_k) \in \Q \times \R^{N_1} \times \R^{N_2} \times \Pi$.
\end{proposition}
\begin{proof}
We immediately see that
\begin{align}
\left |v^\Pi_i(m_{t_k},Y_{t_k},t_k) - v_i^{\check\Pi}(m_{t_k},Y_{t_k},t_k) \right | \leq E_{21} + E_{22} + E_{23}, \notag
\end{align}
where
\begin{align}
E_{21} &= \sup_{\mu \in \A^{Y,\Pi}_{t_k,i}}  \E^{m_{t_k},Y_{t_k},t_k} \biggl [\sum_{l=k}^{N-1} \int_{t_l}^{t_{l+1}}\biggl | \pi^{m_{s}}_{s}\l (f_{\mu_{t_l}}(\cdot, Y_{s}, s)\r)- \pi^{m_{t_l}}_{t_l}\l(f_{\mu_{t_l}}(\cdot, Y_{t_l}, t_l) \r)\biggr |ds \biggr ], \notag \\
E_{22} &= \sup_{\mu \in \A^{Y,\Pi}_{t_k,i}}  \E^{m_{t_k},Y_{t_k},t_k} \biggl [\delta \sum_{l=k}^{N-1}\biggl | \pi^{m_{t_l}}_{t_l}\l (f_{\mu_{t_l}}(\cdot, Y_{t_l}, t_l)\r)- \pi^{\bar m_{t_l}}_{t_l}\l(f_{\mu_{t_l}}(\cdot, \bar Y_{t_l}, t_l) \r)\biggr | \biggr ], \notag \\
E_{23} &=\sup_{\mu \in \A^{Y,\Pi}_{t_k,i}} \E^{m_{t_k},Y_{t_k},t_k} \biggl [\sum_{t_k \leq \tau_n \leq T} \biggl |\pi^{m_{\tau_{n}}}_{\tau_{n}} \l(c_{\xi_{n-1},\xi_{n}}(\cdot, Y_{\tau_{n}},\tau_{n}) \r)\notag\\
 &\qquad\qquad\qquad\qquad\qquad\qquad-\pi^{\bar m_{ \t_{n}}}_{\tau_{n}} \l(c_{\xi_{n-1},\xi_{n}}(\cdot, \bar Y_{\tau_{n}}, \tau_{n}) \r) \biggr | \biggr], \notag
\end{align}
and where $\bar m$, $\bar Y$ denote the Euler discretizations, starting at $m_{t_k},Y_{t_k}$ at $t=t_k$, of $m$ and $Y$, respectively. The rest of the proof is now a combination of the techniques used in the proofs of Lemma \ref{lemma:Lipschitzm} and Proposition \ref{lemma:bermudan},  in combination with standard error estimates for the Euler approximation. To be more precise, by proceeding along the lines of the proof of Proposition \ref{lemma:bermudan} we see
that
\begin{eqnarray*}%\label{jja}
E_{21} \leq c\biggl [\left(\delta\log\l(\frac {2T}{\delta}\r)\right)^{1/2}+\delta\biggr ].
\end{eqnarray*}
To estimate $E_{22}$, let
\begin{align}
 E_{22,l}(\mu)&:=\E^{m_{t_k},Y_{t_k},t_k} \biggl [\biggl | \pi^{m_{t_l}}_{t_l}\l (f_{\mu_{t_l}}(\cdot, Y_{t_l}, t_l)\r)- \pi^{\bar m_{t_l}}_{t_l}\l(f_{\mu_{t_l}}(\cdot, \bar Y_{t_l}, t_l) \r)\biggr | \biggr ]. \notag
\end{align}
Then
$$E_{22,l}(\mu)\leq E_{221,l}(\mu)+E_{222,l}(\mu),
$$
where
\begin{align}
 E_{221,l}(\mu)&:= \E^{m_{t_k},Y_{t_k},t_k} \biggl [\biggl | \pi^{m_{t_l}}_{t_l}\l (f_{\mu_{t_l}}(\cdot, Y_{t_l}, t_l)\r)- \pi^{ m_{t_l}}_{t_l}\l(f_{\mu_{t_l}}(\cdot, \bar Y_{t_l}, t_l) \r)\biggr | \biggr ], \notag \\
 E_{222,l}(\mu)&:=\E^{m_{t_k},Y_{t_k},t_k} \biggl [\biggl | \pi^{m_{t_l}}_{t_l}\l (f_{\mu_{t_l}}(\cdot, \bar Y_{t_l}, t_l)\r)- \pi^{\bar m_{t_l}}_{t_l}\l(f_{\mu_{t_l}}(\cdot, \bar Y_{t_l}, t_l) \r)\biggr | \biggr ]. \notag
\end{align}
Now, using the Lipschitz property of $f_{i}$, and arguing as in \eqref{eq:pilipschitz}, we the see that
\begin{align}
 E_{221,l}(\mu)+E_{222,l}(\mu)\leq c\E^{m_{t_k},Y_{t_k},t_k} \biggl [| m_{t_l}-\bar m_{t_l}|+| Y_{t_l}-\bar Y_{t_l}|\biggr ]. \notag
\end{align}
In particular,
\begin{align}
E_{22} &= \delta\sup_{\mu \in \A^{Y,\Pi}_{t,i}}  \sum_{l=k}^{N-1} E_{22,l}(\mu)\notag\\
&\leq c\delta \E^{m_{t_k},Y_{t_k},t_k} \biggl [\sum_{l=k}^{N-1}\biggl (| m_{t_l}-\bar m_{t_l}|+| Y_{t_l}-\bar Y_{t_l}|\biggr )\biggr ]. \notag
\end{align}
Using standard error estimates for Euler approximations, e.g., see \cite{KP92} Section 10.2, we also have that
\begin{align*}%\label{eq:standardeuler}
\E^{m_{t_k},Y_{t_k},t_k} \biggl [| m_{t_l}-\bar m_{t_l}|+| Y_{t_l}-\bar Y_{t_l}|\biggr ]\leq c\delta^{1/2} \notag
\end{align*}
for all $l\in\{k,\dots,N\}$. In particular,
\begin{align}
E_{22} &\leq c\delta N\delta^{1/2}=cT\delta^{1/2}. \notag
\end{align}
Finally, repeating the argument for $E_{22}$, with $f_i$ replaced by $c_{i,j}$, and invoking Lemma \ref{prop:finitenumberofswitches} we also see that
\begin{align}
E_{23} &\leq c\biggl (\frac {2 T\sup_i \|f_i\|_\infty}{\nu}\biggr )\delta^{1/2} \notag
\end{align}
and hence the proof of the proposition is complete.
\end{proof}

\subsection{Step 3: Controlling $E_3$}
\begin{proposition}\label{lemma:projected} There exists a constant $C_{3}$, independent of $\e$, such that
$$|E_3|=|v_i^{\check\Pi}(m_{t_k},Y_{t_k},t_k)- v_i^{\check\Pi,\e}(m_{t_k},Y_{t_k},t_k)| \leq C_{3} \e,$$
whenever $(i,m_{t_k},Y_{t_k},t_k) \in \Q \times \R^{N_1} \times \R^{N_2} \times \Pi$.
\end{proposition}
\begin{proof} We first note that
\begin{align}
&|v_i^{\check\Pi}(m_{t_k},Y_{t_k},t_k)- v_i^{\check\Pi,\e}(m_{t_k},Y_{t_k},t_k)| \notag \\
&\leq\sup_{\mu \in \A^{Y,\Pi}_{t_k,i}}  \E^{m_{t_k},Y_{t_k},t_k} \bigg [\int_{t_k}^T \l |\pi^{m_{\check s}}_{\check s}\l (f_{\mu_{\check s}}(\cdot, Y_{\check s}, \check s)\r) - \pi^{m^\e_{\check s}}_{\check s} \l (f_{\mu_{s}}(\cdot, Y^\e_{\check s}, \check s)\r ) \r | ds \notag \\
&\qquad\qquad\qquad\qquad + \sum _{t_k\leq \tau_n \leq T} \l | \pi^{m_{\tau_n}}_{\tau_n}\l (c_{{\xi_{n-1}}{\xi_{n}}}(\cdot, Y_{\tau_n},\tau_n)\r) - \pi^{m^\e_{\tau_n}}_{\tau_n} \l (c_{{\xi_{n-1}}{\xi_{n}}}(\cdot, Y^\e_{\tau_n}, \tau_n)\r)  \r| \bigg] \notag
\end{align}
Using a by now familiar argument, based on the Lipschitz property of $f_i$ and $c_{i,j}$, and \eqref{eq:pilipschitz}, we see that
\begin{align}
&\l |\pi^{m_{\check s}}_{\check s}\l (f_{\mu_{\check s}}(\cdot, Y_{\check s}, \check s)\r) - \pi^{m^\e_{\check s}}_{\check s} \l (f_{\mu_{s}}(\cdot, Y^\e_{\check s}, \check s)\r ) \r |\notag\\
&+\l | \pi^{m_{\tau_n}}_{\tau_n}\l (c_{{\xi_{n-1}}{\xi_{n}}}(\cdot, Y_{\tau_n},\tau_n)\r) - \pi^{m^\e_{\tau_n}}_{\tau_n} \l (c_{{\xi_{n-1}}{\xi_{n}}}(\cdot, Y^\e_{\tau_n}, \tau_n)\r)  \r|\notag\\
&\leq c\E^{m_{t_k},Y_{t_k},t_k} \bigg [ |m_t - m^\e_t | +  |Y_t - Y^\e_t |  \biggr ] \notag
\end{align}
for all $s\in[t_k,T]$ and for all $\tau_n$, $t_k\leq \tau_n \leq T$. Hence, using the assumption in \eqref{eq:projection}, and Lemma \ref{prop:finitenumberofswitches},  we can conclude that
\begin{align}
|v_i^{\check\Pi}(m_{t_k},Y_{t_k},t_k)- v_i^{\check\Pi,\e}(m_{t_k},Y_{t_k},t_k)|\leq c\e, \notag
\end{align}
and hence the proof of the proposition is complete.
\end{proof}

\subsection{Step 4: Controlling $E_4$} %\label{subsec:E4}
%Recalling the approximated conditional expectation

\begin{proposition}\label{prop:regression}There exists a  constant $C_4$, independent of $\delta$ and  $\Delta $, such that
\begin{eqnarray*}%\label{Iprin-}
|E_4|=|v_i^{\check\Pi,\e}(m_{t_k},Y_{t_k},t_k)-\hat v_i^{\check\Pi,\e}(m_{t_k},Y_{t_k},t_k)| \leq C_{4} \notag \frac{\Delta}{\delta}
\end{eqnarray*}
whenever $(i,m_{t_k},Y_{t_k},t_k) \in \Q \times \R^{N_1} \times \R^{N_2} \times \Pi$.
\end{proposition}
\begin{proof}
Recall that $\Pi=\Pi^{\delta}=\{0=t_0<t_1<\cdots<t_N=T\}$ with $|t_k-t_{k-1}|=\delta$. We claim that there exist positive constants $A$ and $L$, independent of $\delta$ and  $\Delta$,  such that the following holds. There exists a sequence of constants $\{E_{4,k}\}_{k=0}^N$ such that
\begin{eqnarray}\label{Iprin}
|v_i^{\check\Pi,\e}(m_{t_k},Y_{t_k},t_k)-\hat v_i^{\check\Pi,\e}(m_{t_k},Y_{t_k},t_k)| \leq  E_{4,k}
\end{eqnarray}
whenever $(i,m_1,Y_{t_k},t_k), (i,m_2,Y_{t_k},t_k) \in \Q \times \R^{N_1} \times \R^{N_2} \times \Pi$ and such that
\begin{align}\label{Iprin+}
E_{4,N}&=0\notag\\
E_{4,k}&=A\Delta(1+L\delta)+E_{4,k+1}, \hspace{1cm} 0 \leq k\leq N-1.
\end{align}
We prove \eqref{Iprin}, \eqref{Iprin+} by (backward) induction on $k$ and we say that $I(k)=1$ if \eqref{Iprin} holds with a constant  $E_{4,k}$ whenever $(i,m_{t_k},Y_{t_k},t_k) \in \Q \times \R^{N_1} \times \R^{N_2} \times \Pi$, and if $E_{4,k}$ is related to the bounds $E_{4,k+1}$,\dots, $E_{4,N}$ as stated in
\eqref{Iprin+}.  We want to prove that $I(k)=1$ whenever $k\in\{0,...,N\}$. Since
 $|v_i^{ \check \Pi,\e}(m_T, Y_T,T)-\hat v_i^{ \check \Pi,\e}(m_T, Y_T,T)|=0$ we immediately see that $I(N)=1$. Assuming that
 $I(k+1)=1$ for some $k\in[0,N-1]$ we next prove that $I(k)=1$ by construct $E_{4,k}$. Note that $A$ and $L$ are degrees of freedom appearing in the  following argument and the important thing is, in particular, that $A$ and $L$ do not depend on $k$. To start the proof, recall the conditional expectation $F$ defined in \eqref{eq:regressionsimplenotation} and the approximate conditional expectation $\hat F$  defined in \eqref{eq:Fhat}. The recursive scheme \eqref{eq:scheme2}  can be written as
\begin{align} %\label{eq:vvhatdiff-}
\hat v_i^{\check\Pi,\e}(m_{t_k},Y_{t_k},t_k)
=\max_{j \in \Q}& \bigg \{\delta \pi^{m_{t_{k}}}_{t_k}(f_{j}(\cdot, Y_{t_k},t_k))+ \hat F^{m_{t_{k}},Y_{t_{k}},t_{k}}(\hat v_j^{\check\Pi,\e}) -\pi^{m_{t_{k}}}_{t_k}(c_{i,j}(\cdot, Y_{t_k},t_k))\bigg \}\notag\\
=\max_{j \in \Q}& \bigg \{\delta \pi^{m_{t_{k}}}_{t_k}(f_{j}(\cdot, Y_{t_k},t_k)) + F^{m_{t_{k}},Y_{t_{k}},t_{k}}(v_j^{\check\Pi,\e}) -  \pi^{m_{t_{k}}}_{t_k}(c_{i,j}(\cdot, Y_{t_k},t_k))\notag\\
&+\hat F^{m_{t_{k}},Y_{t_{k}},t_{k}}(v_j^{\check\Pi,\e}) - F^{m_{t_{k}},Y_{t_{k}},t_{k}}(v_j^{\check\Pi,\e})\notag\\
&+\hat F^{m_{t_{k}},Y_{t_{k}},t_{k}}(\hat v_j^{\check\Pi,\e}) - \hat F^{m_{t_{k}},Y_{t_{k}},t_{k}}(v_j^{\check\Pi,\e})\bigg\}. \notag
\end{align}
Hence,
\begin{align} \label{eq:vvhatdiff}
\hat v_i^{\check\Pi,\e}(m_{t_k},Y_{t_k},t_k) &\leq v_i^{\check\Pi,\e}(m_{t_k},Y_{t_k},t_k)\notag\\
&+\max_{j \in \Q} \bigg \{|\hat F^{m_{t_{k}},Y_{t_{k}},t_{k}}(v_j^{\check\Pi,\e}) - F^{m_{t_{k}},Y_{t_{k}},t_{k}}(v_j^{\check\Pi,\e})|\biggr \} \notag \\
&+\max_{j \in \Q} \bigg \{|\hat F^{m_{t_{k}},Y_{t_{k}},t_{k}}(\hat v_j^{\check\Pi,\e}) - \hat F^{m_{t_{k}},Y_{t_{k}},t_{k}}(v_j^{\check\Pi,\e})\bigg \}. 
\end{align}
Using Lemma \ref{lemma:Lipschitzm} and Lemma \ref{lemma:Lipschitzy}, we know that
$$|v_i^{\check\Pi,\e}(m_1,y_1,t_k)-v_i^{\check\Pi,\e}(m_2,y_2,t_k)|\leq (C_{m}+C_{Y})(|m_1-m_2|+|y_1-y_2|)$$
whenever $(i,m_1,y_1,t_k), (i,m_1,y_1,t_k) \in \Q \times \R^{N_1} \times \R^{N_2} \times \Pi$. In particular,
 $v_i^{\check\Pi,\e}(m_t,Y_t,t_k)$ is Lipschitz continuous w.r.t. the $(N_1 + N_2)$-dimensional process $(m_t,Y_t)$ and with Lipschitz constant independent of $k$. Now, using this and Lemma 3.4 in \cite{ACLP12} we can conclude that there exist constants $L$ and $\tilde C_{4,k+1}$, independent of $\delta$, $\Delta $, and $k$,  such that
\begin{equation*}%\label{eq:controlofFandFhat}
|\hat F^{m_{t_{k}},y_{t_{k}},t_{k}}(v_j^{\check\Pi,\e}) - F^{m_{t_{k}},y_{t_{k}},t_{k}}(v_j^{\check\Pi,\e})|
\leq \tilde C_{4,k+1}\Delta(1+L\delta). \notag
\end{equation*}
Essentially $\tilde C_{4,k+1}$ is the Lipschitz constant of $v_i^{\check\Pi,\e}(\cdot,\cdot,t_{k+1})$, i.e., $(C_{m}+C_{Y})$. In particular,
\begin{equation}   \label{eq:controlofFandFhat+}
|\hat F^{m_{t_{k}},y_{t_{k}},t_{k}}(v_j^{\check\Pi,\e}) - F^{m_{t_{k}},y_{t_{k}},t_{k}}(v_j^{\check\Pi,\e})|
\leq A\Delta(1+L\delta)  
\end{equation}
where $A:=c (C_{m}+C_{Y})$ for some harmless constant $c$. Note that $A$ and $L$ are now fixed and, in particular, independent of $k$. Next, by the definition of $\hat F$, see \eqref{eq:Fhat} and \eqref{eq:truemean},
\begin{align}
&|\hat F^{m_{t_{k}},Y_{t_{k}},t_{k}}(\hat v_j^{\check\Pi,\e}) - \hat F^{m_{t_{k}},Y_{t_{k}},t_{k}}(v_j^{\check\Pi,\e})|\notag\\
\leq&\mathbb{E}\biggl [|\hat v_j^{\check\Pi,\e}(m_{t_{k+1}},Y_{t_{k+1}},t_{k+1})-v_j^{\check\Pi,\e}(m_{t_{k+1}},Y_{t_{k+1}},t_{k+1})|\big|(m_{t_k},Y_{t_k})\in B_{r}\biggr ] .\notag
\end{align}
Hence, using this, and the induction hypothesis $I(k+1)=1$, we can conclude that
\begin{equation} \label{eq:controlofFhatandFhata}
|\hat F^{m_{t_{k}},y_{t_{k}},t_{k}}(\hat v_j^{\check\Pi,\e}) - \hat F^{m_{t_{k}},y_{t_{k}},t_{k}}(v_j^{\check\Pi,\e})|\leq E_{4,k+1}.
\end{equation}
Combining \eqref{eq:vvhatdiff}, \eqref{eq:controlofFandFhat+} and \eqref{eq:controlofFhatandFhata}, we see that
$$
\hat v_i^{\check\Pi,\e}(m_{t_k},Y_{t_k},t_k) - v_i^{\check\Pi,\e}(m_{t_k},Y_{t_k},t_k)\leq A\Delta(1+L\delta)+E_{4,k+1}.
$$
By symmetry the same inequality holds for $\hat v_i^{\check\Pi,\e}(m_{t_k},Y_{t_k},t_k) - v_i^{\check\Pi,\e}(m_{t_k},Y_{t_k},t_k)$ and thus
$$
\left|\hat v_i^{\check\Pi,\e}(m_{t_k},Y_{t_k},t_k) - v_i^{\check\Pi,\e}(m_{t_k},Y_{t_k},t_k)\right|\leq E_{4,k},
$$
with $E_{4,k}$ defined as
\begin{equation*}%\label{aa}
E_{4,k}:=A\Delta(1+L\delta)+E_{4,k+1}. 
\end{equation*}
In particular, \eqref{Iprin} and \eqref{Iprin+} hold for $k$ and we have proved that if $I(k+1)=1$ for some $k\in\{0,\dots,N-1\}$, then also $I(k)=1$. Hence
$I(k)=1$ for all $k\in\{0,\dots,N\}$ by induction. Based on \eqref{Iprin} and \eqref{Iprin+} we complete the proof of the proposition by observing that, for any $k\in\{0,\dots,N-1\}$,
$$
E_{4,k} \leq \Delta(1+L\delta)\sum_{l=k+1}^N A\leq c \frac{\Delta}{\delta}.
$$
\end{proof}
%Note that in the above proof, it is possible to get a slightly better constant %if one uses the discrete Gronwall lemma as done in \cite{ACLP12}.

\subsection{Step 5: Controlling $E_5$} %\label{subsec:E5}
\begin{proposition}\label{lemma:truetosample} There exist a constant $C_{51}$,  independent of $\delta$, $\Delta$, $\e$ and $M$, and a constant $C_{52}$,
 independent of $\delta$, $\Delta$, and $M$, such that
\begin{eqnarray*}%\label{finest}
&&\l \|\hat v_i^{\check\Pi,\e}(m_{t_k},Y_{t_k},t_k)-\tilde  v_i^{\check\Pi,\e,M}(m_{t_k},Y_{t_k},t_k) \r \|_{L^2}\notag\\ &&\leq\frac{C_{51}(1+C_{52})}{\delta\sqrt{Mp_{min}(T,\Delta,\e)}}\left(1+\frac{1}{\sqrt{Mp_{min}(T,\Delta,\e)}}\right),
\end{eqnarray*}
whenever $(i,m_{t_k},Y_{t_k},t_k) \in \Q \times \R^{N_1} \times \R^{N_2} \times \Pi$. Here 
\begin{equation*}%\label{eq:pmin}
p_{min}(T,\Delta,\e):=\min_{t\in\Pi^\delta\cap[0,T]}\min_{B_r\subset D^\e}\mathbb P\biggl [(m_{t},Y_t)\in B_r\biggr ]
\end{equation*}
is a strictly positive quantity.
\end{proposition}
\begin{proof} $p_{min}(T,\Delta,\e)$ is strictly positive quantity due to the fact that the domain $D^\e$ is bounded. To prove the proposition we claim that there exist a positive constant $A$, independent of $\delta$, $\Delta$, $\e$ and $M$, and a constant $B$,
 independent of $\delta$, $\Delta$, and $M$, such that the following holds. There exists a sequence of constants $\{E_{5,k}\}_{k=0}^N$ such that
\begin{eqnarray}\label{Iprina}
\l \|\hat v_i^{\check\Pi,\e}(m_{t_k},Y_{t_k},t_k)-\tilde  v_i^{\check\Pi,\e,M}(m_{t_k},Y_{t_k},t_k) \r \|_{L^2} \leq  E_{5,k}
\end{eqnarray}
whenever $(i,m_1,Y_{t_k},t_k), (i,m_2,Y_{t_k},t_k) \in \Q \times \R^{N_1} \times \R^{N_2} \times \Pi$ and such that
\begin{align}\label{Iprin+a}
E_{5,N}&=0\notag\\
E_{5,k}&=\frac{A(1+B)}{\sqrt{Mp_{min}(T,\Delta,\e)}}\left(1+\frac{1}{\sqrt{Mp_{min}(T,\Delta,\e)}}\right)+E_{4,k+1},\ 0 \leq k\leq N-1.
\end{align}
We prove \eqref{Iprina},\eqref{Iprin+a} by (backward) induction on $k$ and we say that $I(k)=1$ if \eqref{Iprina} holds with a constant  $E_{5,k}$ whenever $(i,m_{t_k},Y_{t_k},t_k) \in \Q \times \R^{N_1} \times \R^{N_2} \times \Pi$, and if $E_{5,k}$ is related to the bounds $E_{5,k+1}$,\dots, $E_{5,N}$, as stated in
\eqref{Iprin+a}.  We want to prove that $I(k)=1$ whenever $k\in\{0,\dots,N\}$. Again, we immediately see that $I(N)=1$. Assuming that
 $I(k+1)=1$ for some $k\in\{0, \dots, N-1\}$ we next prove that $I(k)=1$ by constructing $E_{5,k}$. Note that $A$ is a degree of freedom appearing in the following argument and the important thing is, in particular, that $A$ does not depend on $k$. To start the proof, given $i\in\Q$ we in the following let ${\tilde i}^\ast$ 
and $\hat i^\ast$ 
be such that (see \eqref{eq:scheme2} and 
\eqref{eq:scheme3})
\begin{align}
\tilde v_i^{\check\Pi,\e,M}(m_{t_k},Y_{t_k},t_k)=&\delta \pi^{m_{t_{k}}}_{t_k}(f_{\tilde i ^\ast}(\cdot, Y_{t_k},t_k)) + \tilde F_M^{m_{t_k},Y_{t_k},t_k}(\tilde
v_{\tilde i^*}^{\check\Pi,\e,M}) -  \pi^{m_{t_{k}}}_{t_k}(c_{i,\tilde i ^\ast}(\cdot, Y_{t_k},t_k)) \notag \\
\hat v_i^{\check\Pi,\e}(m_{t_k},Y_{t_k},t_k)=&\delta \pi^{m_{t_{k}}}_{t_k}(f_{\hat i ^\ast}(\cdot, Y_{t_k},t_k))  + \hat F^{m_{t_k},Y_{t_k},t_k}(\hat
v_{\hat j^*}^{\check\Pi,\e}) -  \pi^{m_{t_{k}}}_{t_k}(c_{i,\hat i ^\ast}(\cdot, Y_{t_k},t_k)) .\notag
\end{align}
Note that ${\tilde i}^\ast$ and $\hat i^\ast$ are not neccessarily equal, at time $t_k$, for $(m_{t_k}, Y_{t_k})$. Hence, by proceeding as in the proof of Proposition \ref{prop:regression}, we see that
\begin{align}
\tilde v_i^{\check\Pi,\e,M}(m_{t_k},Y_{t_k},t_k)=&\delta \pi^{m_{t_{k}}}_{t_k}(f_{\tilde i^\ast}(\cdot, Y_{t_k},t_k))  + \tilde F^{m_{t_k},Y_{t_k},t_k}_M(\tilde
v_{\tilde i^*}^{\check\Pi,\e,M}) -  \pi^{m_{t_{k}}}_{t_k}(c_{i,\tilde i ^\ast}(\cdot, Y_{t_k},t_k)) \notag\\
=&\Big\{\delta \pi^{m_{t_{k}}}_{t_k}(f_{{\tilde i}^\ast}(\cdot, Y_{t_k},t_k)) + \hat F^{m_{t_k},Y_{t_k},t_k}(\hat
v_{{\tilde i}^\ast}^{\check\Pi,\e}) -  \pi^{m_{t_{k}}}_{t_k}(c_{i,{\tilde i}^\ast}(\cdot, Y_{t_k},t_k))\notag\\
&+\tilde F^{m_{t_k},Y_{t_k},t_k}_M(\hat
v_{{\tilde i}^\ast}^{\check\Pi,\e})-\hat F^{m_{t_k},Y_{t_k},t_k}(\hat
v_{{\tilde i}^\ast}^{\check\Pi,\e})\notag\\
&+\tilde F^{m_{t_k},Y_{t_k},t_k}_M(\tilde
v_{{\tilde i}^\ast}^{\check\Pi,\e,M})-\tilde F^{m_{t_k},Y_{t_k},t_k}_M(\hat
v_{\hat i^*}^{\check\Pi,\e})\Big\}\notag\\
\leq&\hat v_i^{\check\Pi,\e}(m_{t_k},Y_{t_k},t_k)+\max_{j\in{\Q}}\left|\tilde F^{m_{t_k},Y_{t_k},t_k}_M(\hat v_{j}^{\check\Pi,\e})-\hat F^{m_{t_k},Y_{t_k},t_k}(\hat v_{j}^{\check\Pi,\e})\right|\notag\\
&\qquad\qquad\qquad\quad+\max_{j\in{\Q}}\left|\tilde F^{m_{t_k},Y_{t_k},t_k}_M(\tilde v_{j}^{\check\Pi,\e,M})-\tilde F_M^{m_{t_k},Y_{t_k},t_k}(\hat
v_{j}^{\check\Pi,\e})\right|\notag.
\end{align}
In particular, by symmetry we can use this inequality to conclude that
\begin{align} %\label{eq:difftildevandhatv}
\left|\hat v_i^{\check\Pi,\e}(m_{t_k},Y_{t_k},t_k)-\tilde  v_i^{\check\Pi,\e,M}(m_{t_k},Y_{t_k},t_k)\right| \leq E_{51,k}^{\check\Pi,\e,M} + E_{52,k}^{\check\Pi,\e,M}, \notag
\end{align}
for all $i\in\Q$, where
\begin{align}
E_{51,k}^{\check\Pi,\e,M}&=\max_{j\in{\Q}}\left|\tilde F^{m_{t_k},Y_{t_k},t_k}_M(\hat v_{j}^{\check\Pi,\e})-\hat F^{m_{t_k},Y_{t_k},t_k}(\hat v_{j}^{\check\Pi,\e})\right|,\notag\\
E_{52,k}^{\check\Pi,\e,M}&=\max_{j\in{\Q}}\left|\tilde F^{m_{t_k},Y_{t_k},t_k}_M(\tilde v_{j}^{\check\Pi,\e,M})-\tilde F^{m_{t_k},Y_{t_k},t_k}_M(\hat
v_{j}^{\check\Pi,\e})\right|. \notag
\end{align}
In particular,
\begin{align}\label{eq:ineq0p}
&\left\|\hat v_i^{\check\Pi,\e}(m_{t_k},Y_{t_k},t_k)-\tilde  v_i^{\check\Pi,\e,M}(m_{t_k},Y_{t_k},t_k)\right\|_{L^2}\leq \left \| E_{51,k}^{\check\Pi,\e,M}\right\|_{L^2} +  \l \| E_{52,k}^{\check\Pi,\e,M}\right\| _{L^2},
\end{align}
and, hence, we need to find proper bounds for $\left \| E_{51,k}^{\check\Pi,\e,M}\right\|_{L^2}$ and $\left \| E_{52,k}^{\check\Pi,\e,M}\right\|_{L^2}$.
Firstly, to bound $\left \| E_{51,k}^{\check\Pi,\e,M}\right\|_{L^2}$ we can simply apply Lemma 3.6 of \cite{ACLP12} to conclude that
\begin{align}\label{eq:ineq1p}
\left \| E_{51,k}^{\check\Pi,\e,M}\right\|_{L^2} \leq \frac{\tilde A(1+\tilde B)}{\sqrt{Mp_{min}(T,\Delta,\e)}}\left(1+\frac{1}{\sqrt{Mp_{min}(T,\Delta,\e)}}\right),
\end{align}
where $\tilde A$ is independent of $\delta$, $\Delta$, $\e$ and $M$, while $\tilde B$ is independent of $\delta$, $\Delta$, and $M$. Furthermore, $\tilde A$ and $\tilde B$ are independent of $k$. Secondly, to bound $\left \| E_{52,k}^{\check\Pi,\e,M}\right\|_{L^2}$ we first note, by definition \eqref{eq:Ftilde}, that
\begin{align}
&\left|\tilde F_M^{m_{t_k},Y_{t_k},t_k}(\tilde v_{i}^{\check\Pi,\e,M})-\tilde F_M^{m_{t_k},Y_{t_k},t_k}(\hat v_{i}^{\check\Pi,\e})\right|=\sum_{r=1}^R\Gamma_{r,M}B_r(m_{t_k},Y_{t_k},t_k) \notag
\end{align}
where
\begin{align}
\Gamma_{r,M}= \frac{ \dfrac{1}{M}\sum_{\ell = 1}^{M} \biggl ( \tilde v_i^{\check\Pi,\e,M}(m_{t_{k+1}}^\ell,Y_{t_{k+1}}^\ell,t_{k+1}) -\hat v_i^{\check\Pi,\e}(m_{t_{k+1}}^\ell,Y_{t_{k+1}}^\ell,t_{k+1}) \biggr) \mathcal{I}_{   \{(m^\ell_{t_k},Y^\ell_{t_k}) \in B_{r} \}   }   } {\dfrac{1}{M} \sum_{\ell=1}^M \mathcal{I}_{   \{(m^\ell_{t_k},Y^\ell_{t_k}) \in B_{r} \}}}.  \notag
\end{align}
Hence, using the induction hypothesis $I(k+1)=1$ we see that
\begin{align}\label{eq:ineq2p}
\l \|  E_{52,k}^{\check\Pi,\e,M} \r \| _{L^2}\leq E_{5,k+1}.
\end{align}
We now let $A:=\tilde A$, $B:=(1+\tilde B)$ and define  $$E_{5,k}:=\frac{A(1+B)}{\sqrt{Mp_{min}(T,\Delta,\e)}}\left(1+\frac{1}{\sqrt{Mp_{min}(T,\Delta,\e)}}\right)+E_{5,k+1}.$$
Then, using  \eqref{eq:ineq0p}, \eqref{eq:ineq1p} and \eqref{eq:ineq2p} we conclude that
\begin{eqnarray*}%\label{finest+}
\l \|\hat v_i^{\check\Pi,\e}(m_{t_k},Y_{t_k},t_k)-\tilde  v_i^{\check\Pi,\e,M}(m_{t_k},Y_{t_k},t_k) \r \|_{L^2}\leq E_{5,k},
\end{eqnarray*}
whenever $(i,m_{t_k},Y_{t_k},t_k) \in \Q \times \R^{N_1} \times \R^{N_2} \times \Pi$. In particular, \eqref{Iprina} and \eqref{Iprin+a} hold for $k$ and we have proved that if $I(k+1)=1$ for some $k\in\{0,\dots,N-1\}$, then also $I(k)=1$. Hence
$I(k)=1$ for all $k\in\{0,\dots,N\}$ by induction. Based on \eqref{Iprina} and \eqref{Iprin+a} we complete the proof of  Proposition \ref{lemma:truetosample} by observing that, for any $k\in\{0,\dots,N-1\}$,
\begin{eqnarray*}
E_{5,k} &\leq& \sum_{l=k+1}^N \frac{A(1+B)}{\sqrt{Mp_{min}(T,\Delta,\e)}}\left(1+\frac{1}{\sqrt{Mp_{min}(T,\Delta,\e)}}\right)\notag\\
&\leq&T\frac{A(1+B)}{\delta\sqrt{Mp_{min}(T,\Delta,\e)}}\left(1+\frac{1}{\sqrt{Mp_{min}(T,\Delta,\e)}}\right).
\end{eqnarray*}
\end{proof}

\begin{remark}
As can be seen from Propositions \ref{lemma:bermudan} -- \ref{prop:regression}, errors $E_1$ to $E_4$ are obtained in the pathwise sense, because these approximations are constructed for each path of $(m_t,Y_t)$. However, recalling Subsection \ref{sec.sample_mean_approx}, at time $t_k$, the $\mathcal F^Y_{t_k}$-adapted $\hat F$, defined in \eqref{eq:Fhat}, is approximated by $\tilde F_M$ as in \eqref{eq:Ftilde}. The corresponding regression coefficient $\tilde\lambda_{t_k,r}^M$ is a Monte Carlo type approximation of $\hat\lambda_{t_k,r}$ constructed by using $M$ independent trajectories $(m_{t_k}^\ell,Y_{t_k}^\ell)_{\ell=1}^M$ of the process $(\bar m_{t_k},\bar Y_{t_k})$. Therefore, instead of in the pathwise sense, $E_5$  is controlled in the associated $L^2$-norm, at each $t_k$. In addition, since the sample mean approximations are done independently of each other, and backwards in time, the error at time $t_k$ is controlled by the sum of errors at later time steps. This is the statement of Proposition \ref{lemma:truetosample}.
%However, the error in Proposition \ref {lemma:truetosample}, $E_5$, can only be obtained in the $L^2$ sense this error is the result of a from the Monte Carlo sampling for the regression coefficients. Therefore the final error induced by approximation scheme, as stated in Theorem \ref{thm:convergence}, is stated in the $L^2$ sense.
\end{remark}

\subsection{The final proof of Theorem \ref{thm:convergence}} Combining the error bounds obtained in Propositions \ref{lemma:bermudan} -- \ref {lemma:truetosample} with $t_k=0$ proves Theorem \ref{thm:convergence}.

\setcounter{equation}{0} \setcounter{theorem}{0}
\section{A numerical example}
\noindent In this section we present a simplistic numerical example showing some of the features stemming from our set up with only partial information.
In particular, we consider our optimal switching problem under partial information with two modes, $i\in\Q=\{0,1\}$, and
we use the numerical scheme proposed and detailed in the previous sections to estimate $v_i(m_0, \theta_0,Y_0,0)$. To simplify the presentation, we will only focus on the value function starting in mode $1$, i.e., we will here only estimate  $v_1(m_{0},\theta_0,y_0,0)$. We
will in this section slightly abuse notation and let $v_1$ denote both the true value function $v_1$ as well as its approximation $ \tilde v_1^{\check\Pi,\e,M}$. The interpretation should be clear from the context. We consider constants parameters and, if nothing else is specified, the parameters laid forth in Table \ref{tab:parameters} are the ones used in the simulations.

\begin{table}
\begin{center}

    \begin{tabular}{ | l| l | r |l | l | r|}
    \hline
\textbf{Description} & \textbf{Symbol} &\textbf{Value} &\textbf{Description} & \textbf{Symbol} &\textbf{Value} \\ \hline
   Amount of information & G &  1 &\# paths & M  &5000 \\ \hline
Volatility of signal &  C & 1& Time horizon& T  & 1 \\ \hline
Drift of signal&   F & 0 &Time step& $\delta $ & $1/730$ \\ \hline
Cost of opening&   $c_{0,1}$  & 0.01& \# basis functions & R& 100  \\ \hline
Cost of closing & $c_{1,0}$ & $0.001$ & $X_0$  mean &$m_0$& 0  \\   \hline
Running payoff state 1&$f_1(x)$&  $x$ & $X_0$ variance & $\theta_0$ & 0 \\ \hline
 Running payoff state 2& $f_0(x)$ & 0 & $Y_0$ &$y_0$ & 0 \\
\hline
    \end{tabular}
\end{center}
\caption{Parameter values}
\label{tab:parameters}
\end{table}

We emphasize that although the following example may seem simplistic, the numerical scheme proposed can handle general Kalman-Bucy type partial information optimal switching problems. The reason for making a rather crude choice of parameters is twofold. Firstly, using only limited computational power, we wish to keep our numerical model simple to limit the computational time required. With the data presented in Table \ref{tab:parameters}, and with no effort to optimize the code, the average time for a simulation was about two minutes (per grid point) on a standard laptop PC, and it took around 24 hours to generate the figures in this section. Secondly, we think that the specific features of partial information are made more clear with few modes and constant parameters. %Hence to give an idea of the dependence of the value function on the underlying parameters we believe that constant coefficient functions are

\subsection{Comparison with a deterministic method}

Although the numerical scheme proposed here is shown to converge to the true value function \eqref{eq:valuefcnpartial}, the rate of convergence is not clear from the above. Therefore, to check the validity of our model with $5000$ paths, $100$ basis functions and $\delta = 1/730$, we compare the results produced by our numerical scheme with the results in \cite{O14}, where a deterministic method based on partial differential equations is used to solve
the  optimal switching problem under partial information. %\cite{O13} also proves the monotonicity of the value function w.r.t. the amount of information available.
For the parameter values given in Table \ref{tab:parameters}, the relative error between the deterministic solution in \cite{O14} and the numerical scheme proposed here is less then $2.5 \%$, when calculating $v_1$, see Table \ref{tab:comparison}. This indicates that the number of simulated paths, time steps and basis functions considered in this example are large enough to give a reasonable estimate of the value function.

\begin{table}
\begin{center}

    \begin{tabular}{|l |r|r|r|r|}
    \hline
\textbf{$m_0$} & \textbf{PDE method (\cite{O14})} &\textbf{Simulation method} & \textbf{Error} & \textbf{Error} ($\%$ of $v_1$)  \\ \hline
-0.5 & 0.0627  & 0.0621 & 0.0006 & 0.96  \\ \hline
0     &  0.7897 & 0.8078& 0.0181& 2.3 \\ \hline
0.5  &  5.0351& 5.1019 & 0.0668& 1.3\\ \hline
    \end{tabular}
\end{center}
\caption{Comparison between PDE- and simulation methods for different $m_0$.}
\label{tab:comparison}
\end{table}

%Furthermore, when letting the amount of available information vary, i.e., letting $G$ vary, we recover value functions lying between those of no information and full information optimal switching, cf. \cite{}.

\subsection{The influence of information}
We now turn to study the dependence of the value function $v_1$ on the amount of available information, the latter being determined
through the function $G_t$ defining $h$, see \eqref{eq:KB}. Larger $G$ means more information while smaller $G$ means less information. As shown by Figures \ref{fig:mu0runs}, \ref{fig:surfaceplots} and \ref{fig:crosssections} (a), the value function increases as the amount of available information increases. The reason for this behaviour is intuitively clear. When the amount of information increases, it becomes easier for the controller to make the correct decision concerning when to open/close the facility, increasing the possible output of the facility. For comparison, consider the case with no switching costs, i.e., $c_{1,0}=c_{0,1}=0$. It is intuitively clear that the optimal strategy is then to open the facility as soon as the underlying $X_t$ is positive and to close as soon as it becomes negative. If the controller has full information, he knows exactly when the signal switches sign and can thus implement this strategy. However, when the observation of the signal $X_t$ is noisy, the possibility of making correct decisions decreases and, consequently, the value $v_1$ decreases. Note that the case $G=0$ yields $v_1(m_0,0,0,0)=m_0T$, since $G=0$ implies $$E[X_t \, \vline \,\F^Y_t]=E[X_t \, \vline \, \F^Y_0]=m_0=0, \forall t \in [0,T].$$ The limiting cases $G=\infty$ (full information) and $G=0$ (no information) give upper and lower bounds, respectively, for the value function $v_1$.

It can also be seen from Figure \ref{fig:crosssections}~(a) that, for fixed volatility $C$ of the underlying process, the volatility being determined through the function $C_t$ in \eqref{eq:KB}, the value function is concave as a function of $G$. When $G=0$, there is no information about the underlying signal in the observation. In other words, we only observe noise. As mentioned above, the more information about the underlying we know with certainty, the more valuable it is for the decision making. This phenomenon is even more significant when the proportion of the information contained in $Y$ is relatively small compared with the noise (i.e. when $G$ is small). This is because when $G$ is small, increasing $G$ increases the ratio between information and noise (in the observation) much faster than when $G$ i large. In other words, for small $G$ the value function should increase more rapidly.

As $G$ becomes larger, the percentage of information in $Y$ becomes larger, and our problem starts to resemble that of complete information ($G=\infty$). In this case, increasing $G$ by a small amount may not have a noticeable effect in the observation, i.e., the ratio between the underlying and the noise in the observation becomes stable when $G$ is large. This means that the observation $Y$ will not change as significantly as before when $G$ increases, leading to the flattening of the value function $v_1$ seen in Figure \ref{fig:crosssections}~(a).
%
%The discussion above shows that the curve of the value function should be concave in this case, something which is confirmed by the figures presented.

%\textbf{Reformulate!}
%The concavity of $v_1(0,0,0,0)$ as a function of the available information %$G$ (for fixed volatility $C$), see Figure~\ref{fig:crosssections}~(a), %shows that the less information you have available, the more valueable it %is to gain knowledege $X$. Intuitively, this is because the fraction of %noise in the observation decreases rapidly when $G$ increases if $G$ is %small, while for large values of $G$ the fraction of noise decreases more %slowly when $G$ increases (cf. the decay of $\frac{1}{x}$).

\begin{figure}
    \centering
    \includegraphics[scale=0.5]{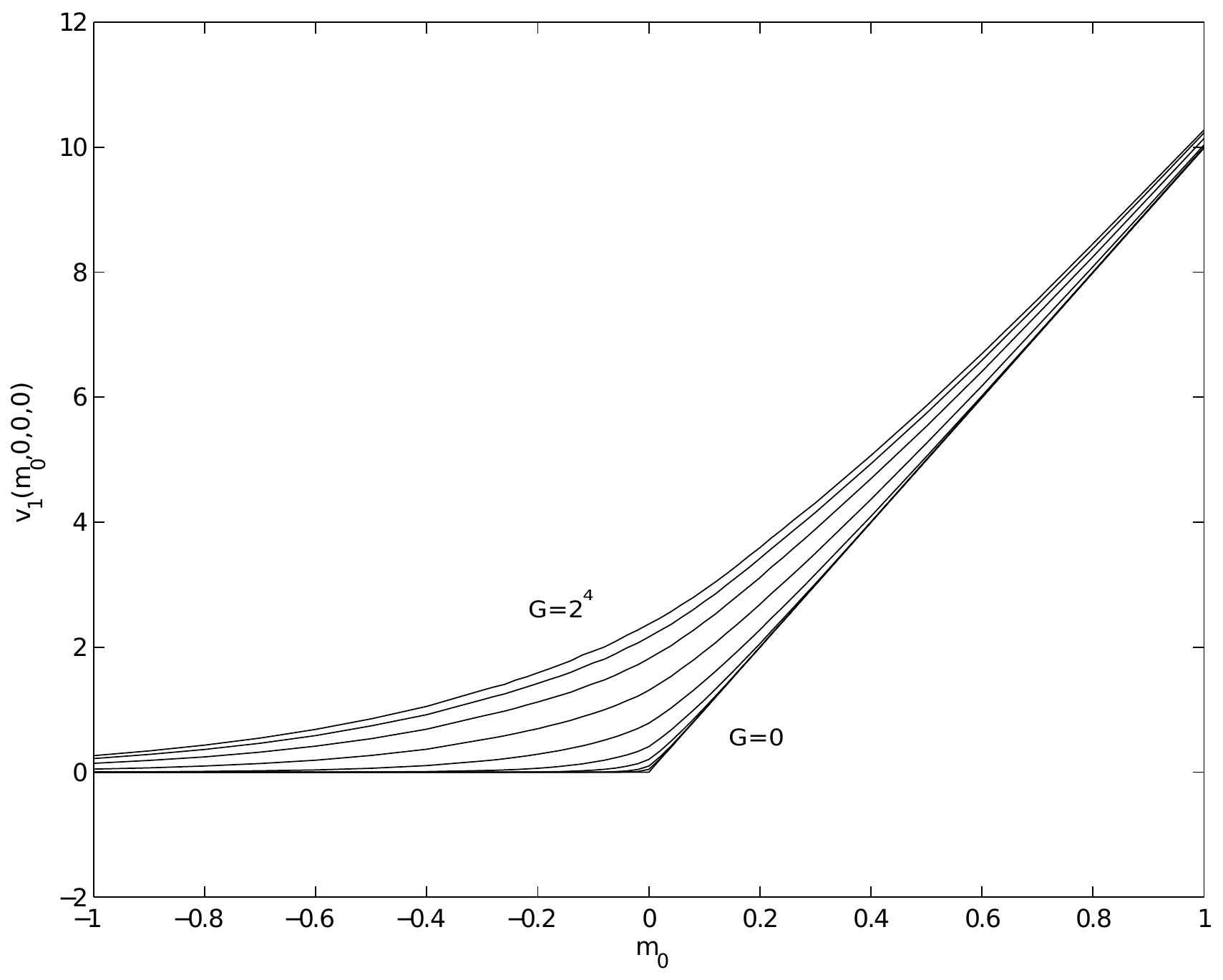}
    \caption{$v_1(m_0,0,0,0)$ for $G \in \{0, 2^{-4}, 2^{-3}, \dots, 2^{4}\}$.}
\label{fig:mu0runs}
\end{figure}

%\fxnote{Should put these in the same plot. Matlab problems...}
\setcounter{figure}{-1}
\begin{figure}
\centering
\begin{minipage}{.5\textwidth}
  \centering
    \includegraphics[width=.9\linewidth]{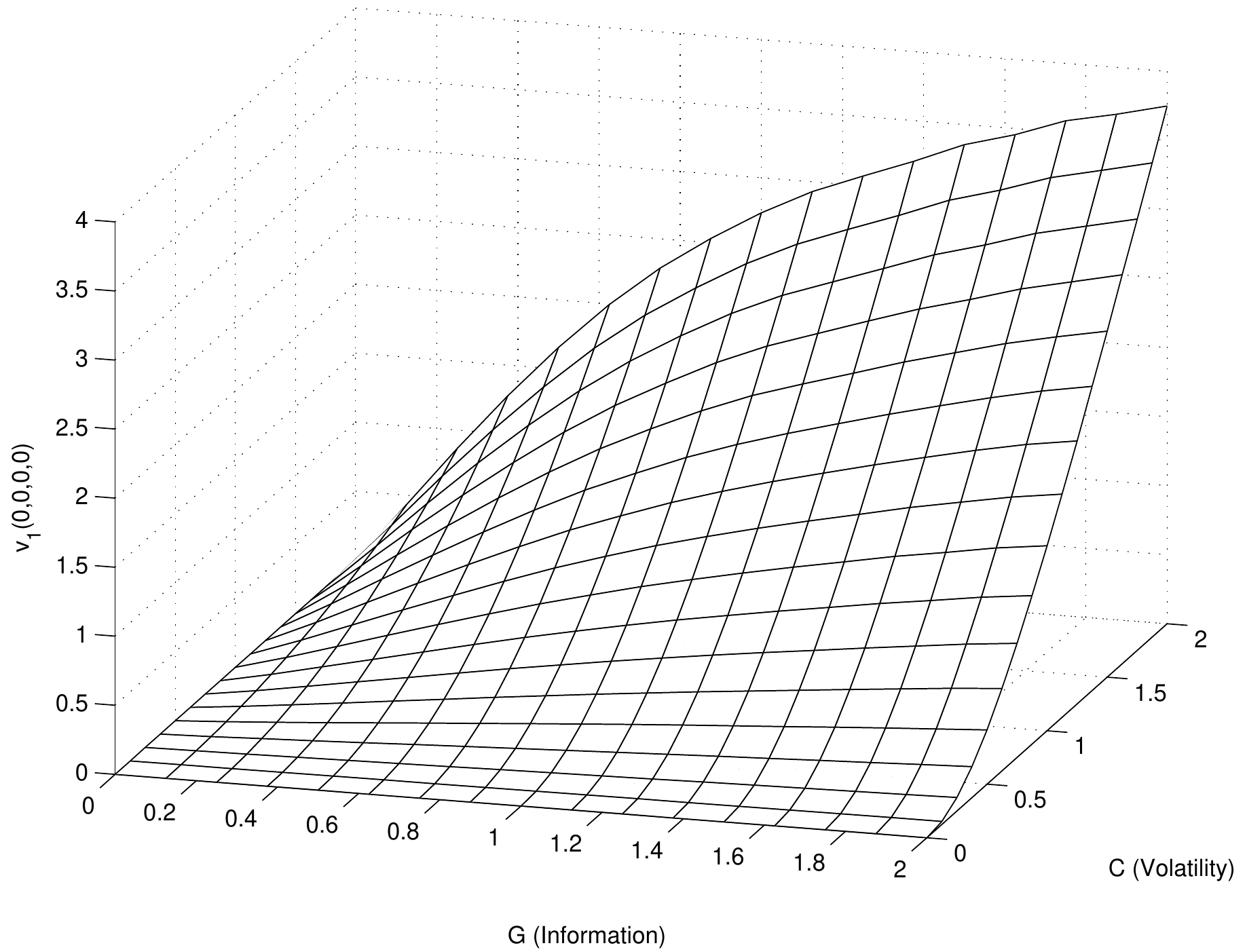}
      \caption{figure}{(a~)$F=0$ }
%\label{fig:surfaceplotF0}
\end{minipage}%
\begin{minipage}{.5\textwidth}
\centering
 \includegraphics[width=.9\linewidth]{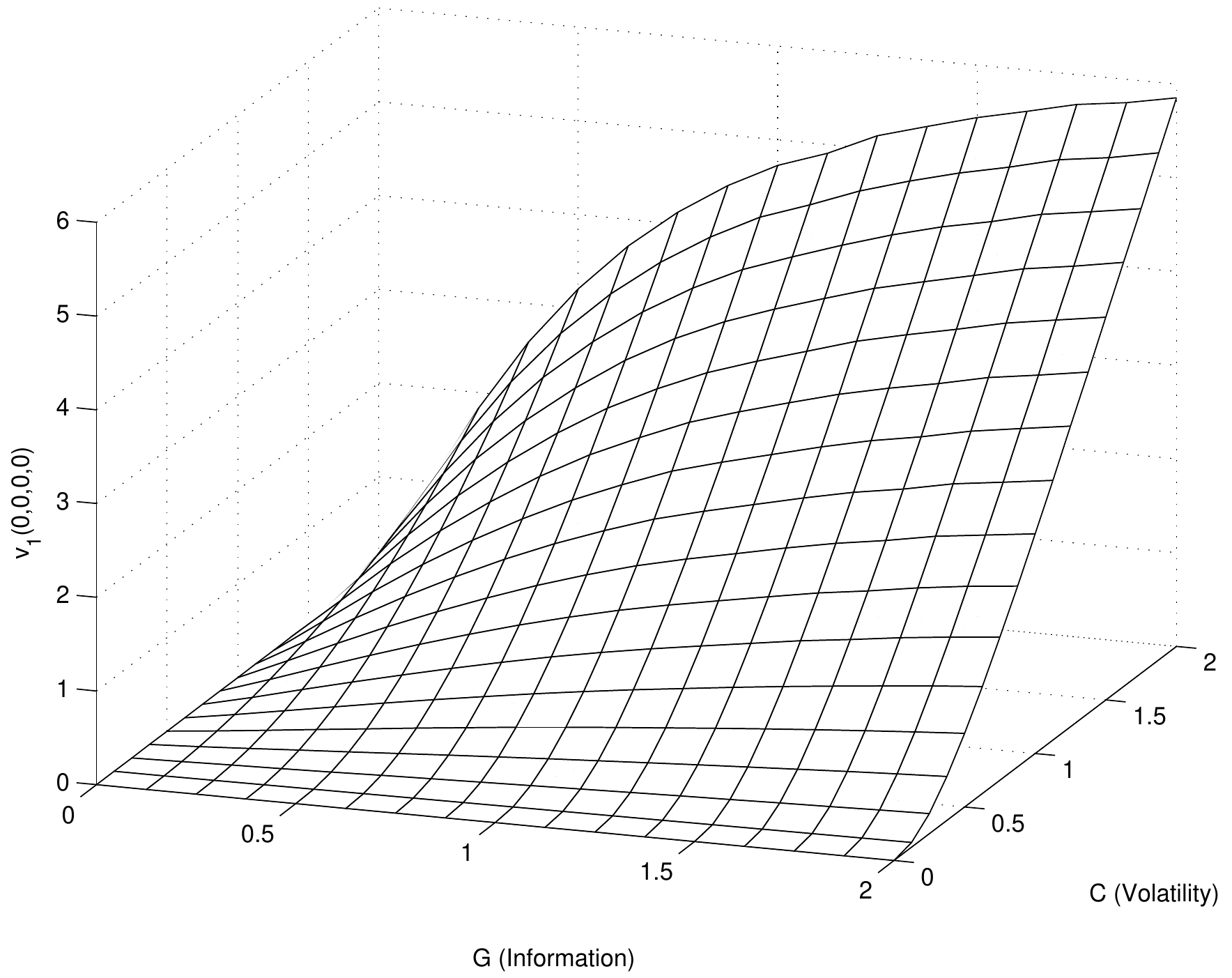}
    \caption{figure}{(b)~$F=1$ }
%\label{fig:surfaceplotF1}
\end{minipage}
\caption{$v_1(0,0,0,0)$ as a function $C$ and $G$.}
\label{fig:surfaceplots}
\end{figure}

\setcounter{figure}{0}
\begin{figure}
\centering
\begin{minipage}{.5\textwidth}
   \centering
  \includegraphics[width=.9\linewidth]{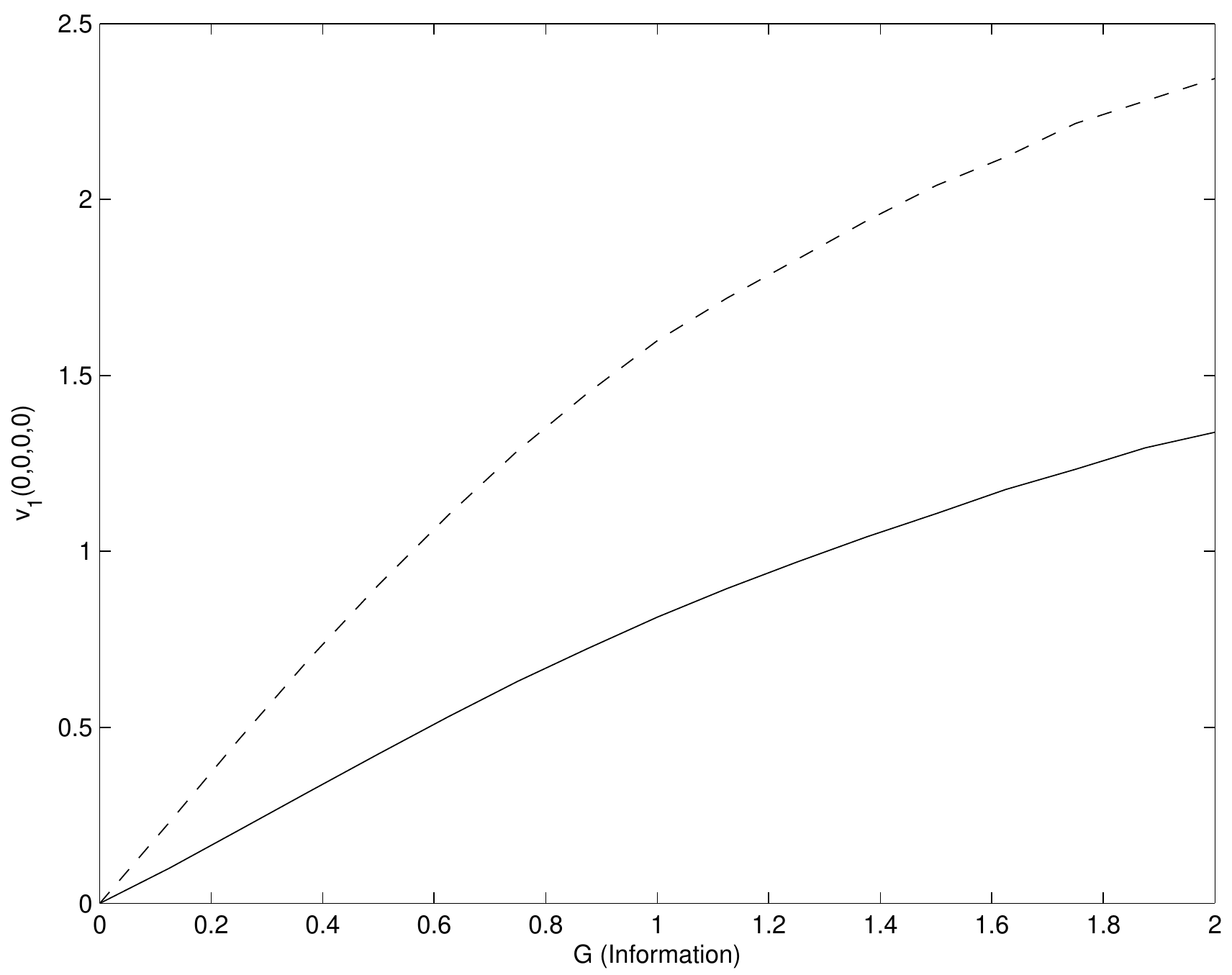}
 \caption{figure}{(a)~$C=1$, }
% \label{fig:c1crosssection}
\end{minipage}%
\begin{minipage}{.5\textwidth}
   \centering
    \includegraphics[width=.9\linewidth]{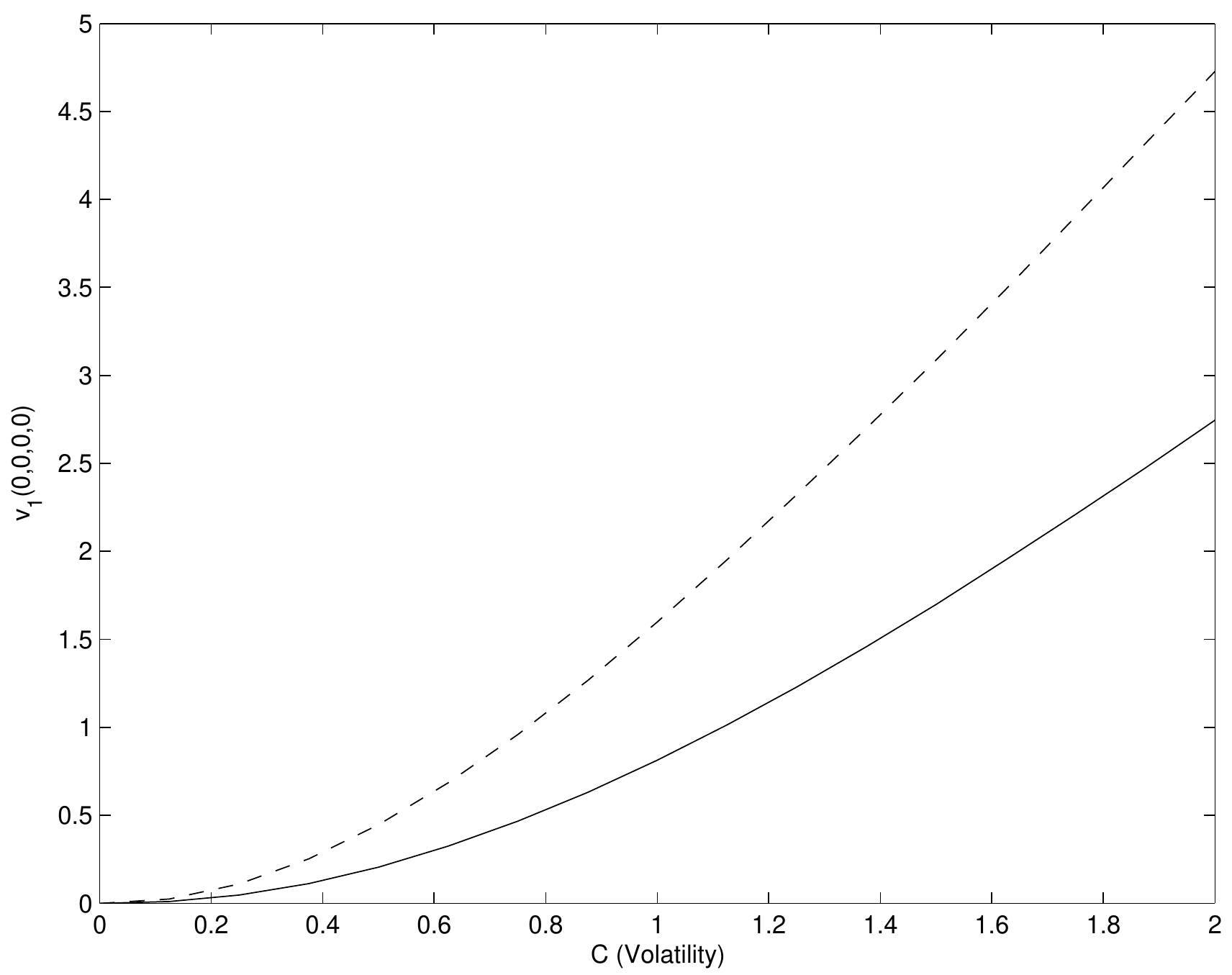}
  \caption{figure}{(b)~$G=1$}
 % \label{fig:g1crosssection}
\end{minipage}
\caption{Cross sections of Figures \ref{fig:surfaceplots} (a) (solid) and \ref{fig:surfaceplots} (b) (dashed).}
\label{fig:crosssections}
\end{figure}

\subsection{The influence of volatility of the underlying signal X}
As for the amount of available information, the value function $v_1$ is monotonically increasing with respect to the volatility of the underlying signal. This behaviour comes from the fact that the possible gain increases when the chance of the signal being high increases, while the risk of making big losses is eliminated by the possibility for the controller to turn the facility off (it is here relevant to compare to the monotonicity of the value function of an American option). In total, this makes the value function increase. Similarly to the above, $C=0$ yields $v_1=0$, since $C=0$ implies $X_t=m_0=0, \forall t \in [0,T]$. This is also confirmed by the figures.

Figure \ref{fig:crosssections} (b) shows the intersections of Figure \ref{fig:surfaceplots} (a) and (b) at $G=1$ and it indicates that $v_1$ is a convex function of $C$ for $G$ fixed. In this case,
the volatility of the underlying $X$ varies, but the proportion of the information of $X$ contained in the observation $Y$ stays unchanged.
$C=0$ means that the underlying $X$ is deterministic, in this case the observation $Y$ is useless because $X$ can be known deterministically from its initial value and its deterministic dynamics.

When we increase $C$, the underlying process $X$ becomes random, which means that we are not able to determine the position of $X$ merely from its initial data. In this case the observation $Y$ begins to play a role as it contains information about the now unknown signal $X$. The value function $v_1$ should become larger as $C$ increases, as the volatility of $X$ can make it grow significantly, and this growth can be exploited by the manager thanks to the information obtained through the process $Y$. The impact of observing $Y$ is not dramatically significant for small values of $C$ as for small $C$, the unknown $X$ is only modestly volatile and can be well estimated by considering only the deterministic part of its dynamics. However, when $C$ continues growing, making $X$ more and more volatile, less and less of the signal $X$ can be known based only on its initial values. The observation $Y$ hence becomes increasingly important as it provides insight on how to optimally manage the facility in a now highly volatile environment. This explains the convexity of the curves in Figure \ref{fig:crosssections}~(b).

\subsection{Influence of the drift of the underlying process}

In the discussion above, the drift $F$ of the underlying signal process is assumed to be $0$. We conclude this numerical example by briefly discussing the impact of the drift of the underlying signal, i.e., what happens when the drift $F \not = 0$. In the case of constant coefficient functions exemplified here, we will assume that $F=1$.  Obviously, a positive drift of the signal should have a positive impact on the value function $v_1$ since $f_1$ is monotonically increasing in $x$. This behavior is also observed in Figure \ref{fig:surfaceplots} (b) and Figure \ref{fig:crosssections}.

\setcounter{equation}{0} \setcounter{theorem}{0}
\section{Summary and future research}
\noindent
In this paper we introduced and studied an optimal switching problem under partial information. In particular, we examined, in detail, the case of a Kalman-Bucy system and we constructed a numerical scheme for approximating the value function. We proved the convergence of our numerical approximation to the true value function using stochastic filtering theory and previous results concerning full information optimal switching problems. Through a numerical example we showed the influence of information on the value function, and after comparison with a deterministic PDE method we could conclude that our numerical scheme gives a reliable estimate of the value function for computationally manageable parameter values.

It should be noted that, although we obtain the complete convergence result for the Kalman-Bucy setting, parts of the analysis does not rely on the linear structure. To the authors' knowledge, the difficulties of extending the complete result to the non-linear setting are twofold. Firstly, we rely heavily on the Gaussian structure and explicit expression of the measure $\pi$ to obtain the Lipschitz continuity of the value function, and the Lipschitz property is crucial in the subsequent convergence proof. For the general non-linear setting, however, we do not explicitly know the distribution $\pi$ and are not able to ensure Lipschitzness. Secondly, apart from the temporal parameter $t$, in the Kalman-Bucy setting the value function depends only on three parameters, $m$, $\theta$ and $Y$. However, in the non-linear setting, we may have to use the particle method to approximate $\pi$ and subsequently construct the value function based on this approximating measure. In this case the number of parameters that the value function will depend on will be at least equal to the number of particles in the approximating measure. Hence, the numbers of parameters  may not be fixed and can be arbitrarily large in this non-linear setting.

The two problems discussed above are challenges when generalizing our analysis to the non-linear setting and naturally lead to interesting directions for future research. To be specific, in future studies we intend to extend our study to the non-linear setting and, in this generalized setting, develop a computationally tractable numerical scheme for estimating the value function.

%An extension where the belief of the covariance matrix is also updated with respect to the information $\F^Y_t$ is possible and leads to non-linear dynamics.

%

\end{document}